\documentclass[a4paper, reqno, 12pt]{amsart}
\pdfoutput=1
\usepackage[utf8]{inputenc} 
\usepackage[T1, T2A]{fontenc}
\usepackage[english]{babel}
\usepackage{amsmath, amssymb, amsfonts, amsthm, amscd, mathrsfs,stmaryrd}
\usepackage{enumitem}
\usepackage[hidelinks]{hyperref}
\usepackage[usenames]{color}
\setlist[enumerate]{label = (\alph*), ref=(\text{\alph*)}}
\setlist[itemize]{nolistsep}

\usepackage{epsfig}
\usepackage{graphicx}
\usepackage[cmtip,arrow]{xy}
\usepackage{pb-diagram, pb-xy}
\usepackage{bm}
\usepackage{tikz}
\usetikzlibrary{calc}
\usetikzlibrary[shapes.geometric]

\renewcommand{\Im}{\mathop{\text{Im}}}
\renewcommand{\phi}{\varphi}
\renewcommand{\ge}{\geqslant}
\renewcommand{\le}{\leqslant}

\newcommand{\KK}{\mathbb{K}}
\renewcommand{\AA}{\mathbb{A}}
\newcommand{\QQ}{\mathbb{Q}}
\newcommand{\ZZ}{\mathbb{Z}}
\newcommand{\GG}{\mathbb{G}}

\newcommand{\TT}{\mathbb{T}}

\newcommand{\zero}{\mathbf{0}}

\def\fg{{\mathfrak g}}

\def\ft{{\mathfrak t}}

\def\fu{{\mathfrak u}}
\newcommand{\NN}{\mathbb{Z}_{>0}}
\newcommand{\Zgezero}{\mathbb{Z}_{\geqslant 0}}

\newcommand{\falpha}{\pmb{\alpha}}
\newcommand{\fbeta}{\pmb{\beta}}

\DeclareMathOperator{\Ker}{Ker}
\DeclareMathOperator{\GL}{GL}

\DeclareMathOperator{\Ad}{Ad}

\DeclareMathOperator{\Aut}{Aut}

\DeclareMathOperator{\id}{id}

\DeclareMathOperator{\Spec}{Spec}
\DeclareMathOperator{\Hom}{Hom}

\DeclareMathOperator{\diag}{diag}

\DeclareMathOperator{\cone}{cone}

\theoremstyle{plain}
\newtheorem{lemma}{Lemma}
\newtheorem{proposition}{Proposition}
\newtheorem{theorem}{Theorem}
\newtheorem{corollary}{Corollary}
\theoremstyle{definition}
\newtheorem{definition}{Definition}

\newtheorem{example}{Example}

\theoremstyle{remark}
\newtheorem{remark}{Remark}

\begin{document}

\title[Affine monoids of corank one]{Affine monoids of corank one}
\author{Yulia Zaitseva}
\address{HSE University, Faculty of Computer Science, Pokrovsky Boulvard 11, Moscow, 109028 Russia}
\email{yuliazaitseva@gmail.com}


\thanks{Supported by the RSF-DST grant 22-41-02019}

\subjclass[2010]{Primary 20M32, 14M25, \ Secondary 14R20, 14R05}

\keywords{Algebraic variety, algebraic group, algebraic monoid, toric variety, group embedding, locally nilpotent derivation, semidirect product, idempotents, center}

\begin{abstract}
We give a classification of noncommutative algebraic monoid structures on normal affine varieties such that the group of invertible elements of the monoid is connected, solvable, and has a one-dimensional unipotent radical. We describe the set of idempotents and the center of such a monoid and give a criterion for existence of the zero element. 
\end{abstract}

\maketitle



\section{Introduction}

An (affine) algebraic monoid is an irreducible (affine) algebraic variety $X$ with an associative multiplication 
$\mu \colon X\times X\to X,\quad (x,y)\mapsto x*y$, 
which is a morphism of algebraic varieties and admits a unity $1\in X$ such that ${1*x=x*1=x}$ for all $x\in X$. The group of invertible elements $G(X)$ of an algebraic monoid $X$ is an algebraic group, which is Zariski open in~$X$. According to~\cite[Theorem~3]{Ri2}, every algebraic monoid $X$ with an affine group of invertible elements $G(X)$ is an affine monoid. For more information on algebraic monoids, see~\cite{Pu1988,Re2005,Ri1,Vi}. 

An affine algebraic monoid $X$ is called reductive (solvable, commutative) if $G(X)$ is a reductive (solvable, commutative) affine algebraic group. The most developed is the theory of reductive monoids, see e.g. the combinatorial classification of reductive monoids in~\cite{Vi,Ri1}. 

By the (co)rank of a monoid $X$ we mean the (co)dimension of a maximal torus in~$G(X)$. Let the ground field $\KK$ be algebraically closed and of characteristic zero. In the commutative case, the group $G(X)$ splits into the direct product of an algebraic torus $(\KK^\times)^r$ and a commutative unipotent group $\GG_a^s$, where $\GG_a = (\KK, +)$ is the additive group of the ground field. 

In~\cite{ABZ2018}, commutative monoids on affine spaces are studied. In particular, \cite[Proposition~1]{ABZ2018} gives a classification of commutative monoid structures of ranks $0$,~$n-1$, and~$n$ on~$\AA^n$. This implies a classification of commutative monoids on~$\AA^1$ and $\AA^2$, and \cite[Theorem~1]{ABZ2018} provides a classification on~$\AA^3$. In~\cite{GvSe2022}, these classifications are extended to non-closed ground fields of characteristic zero. 

In~\cite{DzZa2021}, a classification of commutative monoid structures of rank $0$, $n-1$, and~$n$ on normal affine varieties is obtained. It turns out that such an affine algebraic variety is toric, and structures of corank one are described by Demazure roots of the variety. The classification is given in terms of comultiplications $\mu^*\colon \KK[X] \to \KK[X] \otimes \KK[X]$:
\[\chi^u \mapsto \chi^u \otimes \chi^u (1 \otimes \chi^e + \chi^e \otimes 1)^{\langle p, u\rangle},\] 
where $p$ is the primitive vector on a ray of the cone of $X$ and $e$ is a Demazure root corresponding to this ray, see Section~\ref{prelim_sec} for definitions. In particular, it implies the classification of commutative monoid structures of any rank for normal affine surfaces. 

In~\cite{Bi2021}, a classification of noncommutative monoids on affine surfaces is given. In this case, $G(X)$ is a semidirect product of $\KK^\times$ and $\GG_a$. It is proved that the surface is toric and the comultiplication has the form
\begin{equation}
\label{comult_intr_eq}
\chi^u \mapsto \chi^u \otimes \chi^u (1 \otimes \chi^{e_1} + \chi^{e_2} \otimes 1)^{\langle p, u\rangle},
\end{equation}
where $e_1, e_2$ are two Demazure roots corresponding to the same ray of the cone of~$X$ with primitive vector~$p$. 

In this work, we obtain a classification of noncommutative affine monoids of corank one on normal varieties of an arbitrary dimension. This result generalizes classifications in~\cite{DzZa2021} and~\cite{Bi2021}. Let $X$ be such a monoid. In Proposition~\ref{toric_theor} we show that $X$ is toric. It was proved earlier in~\cite{Bi2021} for affine monoids of dimension~2. 
Then we get a classification of monoid structures of corank one in terms of comultiplications and toric geometry in Theorem~\ref{classif_theor}, see Section~\ref{proofclass_sec} for a proof. It turns out that formula~\eqref{comult_intr_eq} works in arbitrary dimension. Theorem~\ref{affspace_theor} is the specialization of the classification to affine spaces. In Theorem~\ref{idemp_theor}, we obtain a description of idempotents in~$X$. In contrast to the commutative case, the number of idempotents may be infinite under some conditions on Demazure roots~$e_1, e_2$ and the cone~$\sigma$ of the toric variety. In Proposition~\ref{idemp_geom_prop}, we study geometry of the subvariety~$E(X)$ of all idempotents in~$X$. In particular, we show that any irreducible component of $E(X)$ is either isomorphic to the affine line or is an isolated point. In Proposition~\ref{zero_prop}, it is proved that the monoid~$X$ admits the zero element if and only if $\sigma^\perp=0$ and $-e_1,-e_2\notin\sigma^\vee$. 
In Proposition~\ref{centerX_prop}, the system of equations defining the center $Z(X)$ of~$X$ is obtained. It follows that $\dim Z(X) = \dim X-2$; moreover, all irreducible components have the maximal dimension if $X$ is an affine space; see Corollaries~\ref{dimZ_cor} and~\ref{centerAn_cor}. Finally, the connection between irreducible components of~$E(X)$ and the center is given in Proposition~\ref{idempcenter_prop}. 

The author is grateful to Ivan Arzhantsev for valuable suggestions, attention to this work and permanent support, and to Roman Avdeev for useful comments. Special thanks are due to the anonymous reviewer for substantial corrections and remarks. 


\section{Semidirect products and the toric structure}

\begin{definition}
An (affine) irreducible normal algebraic variety~$X$ with a morphism ${X \times X \to X, \, (x,y) \mapsto x*y}$, is called an \emph{(affine) algebraic monoid} if $x*(y*z) = (x*y)*z$ for all $x,y,z \in X$ and there exists a point $1 \in X$ called \emph{unity} such that $1 * x = x * 1 = x$. 
\end{definition}

Let $X$ be an affine monoid of dimension~$n$. Then the group~$G(X)$ of invertible elements of~$X$ is a connected affine algebraic group, which is Zariski open in~$X$, see~\cite[Theorem~1]{Ri1} and~\cite[Theorem~5]{Ri2}. Recall that the \emph{rank} and the \emph{corank} of an affine algebraic group~$G$ are the dimension and the codimension of a maximal torus in~$G$, respectively. By the rank and the corank of a monoid~$X$ we mean the rank and the corank of the group $G(X)$, respectively. 

The aim of the present work is to study affine monoids of corank one. In this case, $G(X)$ has no semisimple part since any root subgroup in a semisimple group occurs together with an opposite one. 
So $G(X)$ has a one-dimensional unipotent radical and $G(X) = \GG_a \leftthreetimes T$, where $\GG_a = (\KK,+)$ and $T$ is a torus of dimension~$n-1$. 

By definition of a semidirect product, the group operation in~$G(X)$ is defined by a homomorphism $\chi\colon T \to \Aut(\GG_a) = \KK^\times$:
\begin{equation} \label{Gchi_mult_eq}
(\alpha, t] \cdot (\alpha', t'] = (\alpha + \chi(t)\alpha', \, tt'],
\end{equation}
where $\alpha, \alpha' \in \GG_a$ and $t, t' \in T$. 
We denote this group $G_\chi$. Thus 
\[G_\chi = \GG_a \leftthreetimes T,\]
where $T$ is an algebraic torus of dimension~$n-1$ and the multiplication is defined by~\eqref{Gchi_mult_eq}.

Notice that according to~\eqref{Gchi_mult_eq} we have $(\alpha, t] = (\alpha, 1](0, t] = \alpha \cdot t$, where $\alpha\in \GG_a$, $t \in T$. There is also a dual way to write the elements of~$G_\chi$. Namely, for an element $(\alpha, t] \in G_\chi$ consider $\beta = \chi^{-1}(t)\alpha \in \GG_a$. Then $t \cdot \beta = (0, t](\beta, 1] = (\chi(t)\beta, t] = (\alpha, t]$. Thus for any element $g \in G_\chi$ we have two decompositions
\[g = (\alpha, t] = \alpha \cdot t = [t, \beta) = t \cdot \beta.\]
The dual multiplication rule is
\[[t, \beta) \cdot [t', \beta') = [tt', \chi(t')^{-1}\beta + \beta').\]

The group $G_\chi$ is commutative if and only if $\chi=0$. Let us denote the center of a group~$G$ by $Z(G)$.

\begin{lemma}
\label{centerG_lem}
If $\chi \ne 0$, then $Z(G_\chi) = \{0\} \times \Ker \chi$. 
\end{lemma}

\begin{proof}
Let $(\alpha, t] \in Z(G_\chi)$, i.e. 
\begin{equation} \label{centerGchi_eq}
(\alpha + \chi(t)\alpha', \, tt'] = (\chi(t')\alpha + \alpha', \, t't]
\end{equation}
for any $\alpha' \in \GG_a$ and $t' \in T$. 
Since there exists $t' \in T$ with $\chi(t') \ne 1$, condition~\eqref{centerGchi_eq} for $\alpha'=0,1$ implies $\alpha = 0$ and $\chi(t) = 1$. Conversely, for $(\alpha, t] \in \{0\} \times \Ker \chi$ equation~\eqref{centerGchi_eq} turns into $(\alpha', \, tt'] = (\alpha', \, t't]$. 
\end{proof}

\begin{definition}
A \emph{group embedding} of an algebraic group $G$ is an irreducible algebraic variety~$X$ with an open embedding $G \hookrightarrow X$ such that the action of $G \times G$ on~$G$ by left and right multiplication can be extended to an action of~$G\times G$ on~$X$. 
\end{definition}

If $X$ is a monoid, then $G(X) \hookrightarrow X$ is a group embedding. For affine monoids the converse is also true, i.e. if $X$ is an affine group embedding of~$G$, then the multiplication on~$G$ extends to a multiplication $X \times X \to X$ in such a way that $G$ is the group $G(X)$ of invertible elements of~$X$, see~\cite[Theorem~1]{Vi} for characteristic zero and~\cite[Proposition~1]{Ri1} for the general case. 

\medskip

Recall that a normal irreducible algebraic variety is called \emph{toric} if it admits an effective action of an algebraic torus $\TT$ with an open orbit, see~\cite{CLS2011,Fu1993,Od1988}. In other words, a toric variety is a group embedding of an algebraic torus. The main result of this section is the following statement. 

\begin{proposition}
\label{toric_theor}
Any affine monoid $X$ of corank one is a toric variety. Moreover, the group of invertible elements $G(X)$ is invariant with respect to the acting torus $\TT$. 
\end{proposition}

\begin{proof}
If $\chi = 0$, then $G_\chi$ is commutative and the result follows from~\cite[Theorem~2 and Lemma~2]{ArKo2015}. 

Let $\chi \ne 0$. First consider the actions of $G_\chi$ on itself by left and right multiplication. Let a group homomorphism $\theta\colon T \times T \to \Aut(G_\chi)$ be defined by 
\[\theta(t_1, t_2)(g) = t_1gt_2^{-1}.\]
An element $(t_1, t_2) \in T \times T$ belongs to the kernel of~$\theta$ if and only if $(t_1, t_2)$ belongs to the diagonal of the center $\diag Z(G_\chi) = \{(g,g) \mid g \in Z(G_\chi)\} \subseteq G_\chi \times G_\chi$. By Lemma~\ref{centerG_lem}, $Z(G_\chi) = \{0\} \times \Ker \chi \subseteq \GG_a \leftthreetimes T$, so $\Ker \theta = \diag \Ker \chi$. 
Since $T \times T$ is a torus, the image \[\TT = \Im \theta \subseteq \Aut(G_\chi)\] is a torus as well. It is isomorphic to $(T \times T)/\diag\Ker\chi$, so the torus $\TT$ has dimension $\dim \TT = (n-1) + (n-1) - (n-2) = n$. 

Thus we have an effective action of $\TT$ on $G_\chi$ coming from left and right multiplications of~$G_\chi$ on itself. By definition of a group embedding, this action can be extended to the action on~$X$, so there exists an effective action of $\TT$ on~$X$. It is known that any effective action of a torus of dimension coinciding with the dimension of a variety has an open orbit, see~\cite[Corollaire~1, P.~521]{De1970}. We have $\dim \TT = \dim G_\chi = \dim X = n$, so both actions of $\TT$ on $G_\chi$ and $X$ have an open orbit. 
\end{proof}

The idea of the following generalization of Proposition~\ref{toric_theor} was proposed by Sergey Gorchinskiy and Constantin Shramov. Consider a solvable monoid $X$, i.e. a monoid with a solvable group of invertible elements $G(X)$. Then $G(X) = U \leftthreetimes T$, where $T$ is a torus and $U$ is the unipotent radical of~$G(X)$. In some cases it can be proved that $X$ is toric as well. 

Namely, let the multiplication in a semidirect product $G = U \leftthreetimes T$ be given by a homomorphism $\psi\colon T \to \Aut(U)$:
\[(u,t] \cdot (u',t'] = (u \cdot \psi(t)(u'), tt'], \quad\quad u,u' \in U, \;\; t,t' \in T.\]

\begin{definition}
This semidirect product is called \emph{active} if $\dim T + \dim \Im \psi = \dim G$ or, equivalently, $\dim \Im \psi = \dim U$. 
\end{definition}

\begin{example}
Any connected affine algebraic group of corank one is solvable, and the semidirect product $G_\chi = \GG_a \leftthreetimes T$ is active if and only if $\chi \ne 0$. 
\end{example}

Since $U$ is a unipotent group, it is isomorphic to a vector space as an algebraic variety. Let us prove that we can choose coordinates in $U$ in such a way that $T$ acts on $U$ linearly. 

\begin{lemma}
There are coordinates in $U$ such that $\psi(T) \subseteq \GL(U)$. 
\end{lemma}

\begin{proof}
For $G = U \leftthreetimes T$ and its Lie algebra $\fg = \fu \oplus \ft$, consider the conjugation 
\[\psi\colon G \to \Aut(G), \quad \psi(g)(g') = gg'g^{-1},\]
and the adjoint representation
\[\Ad\colon G \to \Aut(\fg), \quad \Ad(g) = d\psi(g).\] 
Since $U$ is normal in $G$, we have restrictions
\[\psi\colon T \to \Aut(U) \quad\text{and}\quad \Ad\colon T \to \Aut(\fu) \subseteq \GL(\fu).\]
It is known that the exponential map $\exp\colon \fu \to U$ is bijective and commutes with the differential, i.e.
$\exp \circ \,\psi(t) = \Ad(t) \circ \exp$ for any $t \in T$. So $\exp$ takes coordinates on $\fu$ to~$U$ in such a way that $\psi(t) \in \GL(U)$ for any $t \in T$. 
\end{proof}

So one may assume that $\psi(t) = \diag(\chi_1(t), \ldots, \chi_k(t)) \in \GL(U)$, where $\chi_i, \, 1 \le i \le k$, are characters of the torus~$T$. In these terms, the semidirect product is active if and only if $\chi_i$ are linearly independent. 

\begin{example}
Notice that $\dim \Im \psi \le \dim T$, so for an active semidirect product there is a necessary condition $\dim T \ge \frac12\dim G$. For example, the solvable group $B$ of upper triangular $n \times n$ matrices is the semidirect product of the group of unitriangular matrices $U$ and the torus of diagonal matrices~$T$. It is not active for $n > 3$ since $\dim T = n < \frac12 \dim G = \frac{n(n+1)}{4}$. For $n=3$ we have $\dim T = 3 = \frac12\dim B$, however, the semidirect product is not active as well since $\psi(\diag(t_1, t_2, t_3))$ acts on $U$ with linear dependent characters $t_1t_2^{-1}, \, t_1t_3^{-1}, \, t_2t_3^{-1}$. 
\end{example}

\begin{example}
The group \[G = \left\{\begin{pmatrix}t_1 & 0 & \ldots & 0 & u_1\\ 0 & t_2 & \ldots & 0 & u_2\\ \ldots & \ldots & \ddots & \ldots & \ldots\\ 0 & 0 & \ldots & t_{n-1} & u_{n-1} \\ 0 & 0 & \ldots & 0 & t_n\end{pmatrix}, \, t_i \in \KK^\times, u_j \in \KK\right\}\] is an active semidirect product of the group $U$ of unitriangular matrices of~$G$ and the torus of diagonal matrices~$T$ since $\psi(T) \subseteq \GL(U)$ is generated by linearly independent characters $t_1t_n^{-1}, \, \ldots, t_{n-1}t_n^{-1}$. 
\end{example}

\begin{proposition}
Any solvable affine monoid $X$ with active group of invertible elements is a toric variety. Moreover, the group of invertible elements $G(X)$ is invariant with respect to the acting torus $\TT$. 
\end{proposition}

\begin{proof}
Following the proof of Proposition~\ref{toric_theor}, consider the map $\theta\colon T \times T \to \Aut(G)$ and the torus $\TT = \Im \theta$. An element $(t_1, t_2) \in T \times T$ belongs to $\Ker \theta$ if and only if $t_1 = t_2 \in Z(G)$. Notice that $t = (0, t] \in Z(G)$ if and only if $t \in \Ker \psi$. Indeed, if $(\psi(t)(u'), \, tt'] = (u', \, t't]$ for any $u' \in U$, $t \in T$, then $\psi(t) = \id$. So $\Ker\theta = \diag \Ker \psi$. 

Since $\dim T + \dim \Im \psi = \dim G$, the torus $\TT$ has dimension 
\[\dim(T \times T) - \dim \Ker\theta = \dim T + \dim T - (\dim T - \dim \Im \psi) = \dim G = \dim X.\] 
As above, it follows that the action of~$\TT$ has an open orbit on~$G(X)$ and~$X$.
\end{proof}

It would be interesting to investigate affine algebraic monoids with an arbitrary active group of invertible elements, and we plan to do it in a further publication. In this article, we concentrate on affine monoids of corank one. 

\section{Preliminaries on toric varieties}
\label{prelim_sec}

\subsection{Polyhedral cones of an affine toric variety} \label{cones_subsec}
Let us recall basic facts on affine toric varieties; see, for example,~\cite[Section~1.3]{Fu1993}. 
Let $N$ be a lattice of rank $n$, $N_\QQ = N \otimes_\ZZ \QQ$ be the rational vector space of dimension~$n$, and $\sigma \subseteq N_\QQ$ be a strongly convex polyhedral cone. Below we give a construction of an affine toric variety $X_\sigma$ of dimension $n$ corresponding to the cone~$\sigma$. It is known that all affine toric varieties arise in this way. 

Let $M = \Hom(N, \ZZ)$ be the dual lattice, $M_\QQ = M \otimes_\ZZ \QQ$ be the corresponding rational vector space, and $\langle\,\cdot\,,\,\cdot\,\rangle\colon N_\QQ \times M_\QQ \to \QQ$ be the natural pairing. We can consider $M$ as the lattice of characters of a torus $\TT = \Hom(M, \KK^\times)$ of dimension~$n$; for a lattice element $u \in M$, let $\chi^u\colon \TT \to \KK^\times$ be the corresponding character. Then the dual lattice $N$ is identified with the lattice of one-parameter subgroups of the torus~$\TT = N \otimes_\ZZ \KK^\times$. Consider the polyhedral cone $\sigma^\vee$ dual to the cone~$\sigma$:
\[
\sigma^\vee = \{u \in M_\QQ \mid \langle v, u\rangle \ge 0 \; \text{ for all } v \in \sigma\}.
\]
Then $\sigma^\vee$ is a cone of full dimension, $S_\sigma = M \cap \sigma^\vee$ is a finitely generated semigroup with $\ZZ S_\sigma = M$, and $\KK[S_\sigma] = \!\bigoplus\limits_{u \in S_\sigma} \!\KK\chi^u$ is a finitely generated $\KK$-algebra. 
We define $X_\sigma$ as the spectrum of the algebra $\KK[S_\sigma]$. 

Consider $\KK$ as a semigroup with respect to multiplication. Any semigroup homomorphism $S_\sigma \to \KK$ defines an algebra homomorphism $\KK[X_\sigma] \to \KK$, which corresponds to a point in~$X_\sigma$. Conversely, any point in~$X_\sigma$ is defined by some semigroup homomorphism ${S_\sigma \to \KK}$. The torus $\TT$ can be identified with the subset $X_0$ of $X_\sigma$ as the restrictions of group homomorphisms $M \to \KK^\times$ to semigroup homomorphisms $S_\sigma \to \KK$. 

In fact, $X_0$ is open in~$X_\sigma$, so any character $\chi^u$ of the torus $\TT$ can be identified with a rational function on~$X_\sigma$. We have the decomposition
\begin{equation}
\label{torgrad_eq}
\KK[X_\sigma] = \!\bigoplus\limits_{u \in S_\sigma} \!\KK\chi^u. 
\end{equation}
Let an element $t$ of the torus $\TT$ act on $\chi^u$ by multiplication by $\chi^u(t)$. This defines the action of~$\TT$ on $\KK[X_\sigma]$ and hence the action of~$\TT$ on~$X_\sigma$. 

\subsection{Faces and toric open subsets} \label{Urho_subsec}
Recall that a face of a cone is the intersection of the cone with a supporting hyperplane. There is a bijection between faces of the dual cones~$\sigma$ and~$\sigma^\vee$: for any $k$-dimensional face $\tau$ of the cone~$\sigma$ the set 
\[\tau^\perp \cap \sigma^\vee = \{u \in M_\QQ \mid \langle v, u\rangle = 0 \;\;\forall v \in \tau\} \cap \sigma^\vee\]
is an $(n-k)$-dimensional face of~$\sigma^\vee$, and, conversely, for any face $\gamma$ of the cone~$\sigma^\vee$ of dimension~$k$ the set $\gamma^\perp \cap \sigma$ is a face of~$\sigma$ of dimension $n-k$. Faces of dimension one are called \emph{rays}, and faces of codimension one are called \emph{facets}. 

Any homomorphism of semigroups $S_\sigma \to S_{\sigma'}$ defines a homomorphism $\KK[S_\sigma] \to \KK[S_{\sigma'}]$ of algebras and a morphism $\Spec \KK[S_{\sigma'}] \to \Spec \KK[S_\sigma]$. 
In particular, if $\tau \subseteq \sigma$, then $S_\tau \supseteq S_\sigma$, which determines a morphism $X_\tau \to X_\sigma$. It can be proved that the morphism $X_\tau \to X_\sigma$ is an open embedding if and only if $\tau$ is a face of the cone~$\sigma$. Moreover, in this case $S_\tau = S_\sigma + \Zgezero(-u')$, where $u' \in M$ and $\tau = \sigma \cap u'^\perp$, or, equivalently, $u' \in M$ is from the relative interior of the dual face $\tau^\perp \cap \sigma^\vee$ of the cone~$\sigma^\vee$. Then $\KK[S_\tau]$ is the localization of $\KK[S_\sigma]$ at $\chi^{u'}$, which determines an embedding $X_\tau \subseteq X_\sigma$ as a principal open subset. In terms of semigroup homomorphisms, $X_\tau$ consists of points $S_\sigma \to \KK$ such that the image of~$u'$ is nonzero, or, equivalently, the image of any $u' \in M$ from the relative interior of $\tau^\perp \cap \sigma^\vee$ is nonzero. In particular, for all faces $\tau$ of~$\sigma$, the actions of $\TT$ on~$X_\tau$ defined in subsection~\ref{cones_subsec} are compatible and have the open orbit is~$X_0$. 

\subsection{Torus orbits} \label{Torb_subsec}
It is known that $\TT$-orbits on $X_\sigma$ correspond to faces of the cone~$\sigma$, see~\cite[Section 3.1]{Fu1993}. More precisely, for a face $\tau \subseteq \sigma$ denote by $x_\tau$ the distinguished point in $X_\sigma$ given by the following semigroup homomorphism $S_\sigma \to \KK$:
\[u \mapsto \begin{cases}1 \;\text{ if } u \in \tau^\perp,\\
0 \;\text{ otherwise}.\end{cases}\]
By $O_\tau$ we denote the $\TT$-orbit of the point~$x_\tau$; it has dimension $n-\dim\tau$. In terms of semigroup homomorphisms it consists of homomorphisms $S_\sigma \to \KK$ such that
\[u \mapsto \begin{cases}\KK^\times \;\text{ if } u \in \tau^\perp,\\
0 \;\text{ otherwise}.\end{cases}\]
It is known that all orbits of $\TT$ on $X_\sigma$ are of the form $O_\tau$ for a face $\tau$ of $\sigma$. If $\gamma$ is a face of~$\tau$, then $O_{\tau} \subseteq \overline{O_\gamma}$. For more details, see \cite[Theorem 3.2.6]{Cox1995}. 

\subsection{Locally nilpotent derivations, $\GG_a$-actions, and Demazure roots}
\label{LND_Dem_subsec}
For an algebra~$A$, a linear operator $\delta\colon A \to A$ is a \emph{derivation} if it satisfies the Leibniz rule: \[\delta(fg) = \delta(f)g + f \delta(g)\] for any $f, g \in A$. A derivation $\delta$ is called \emph{locally nilpotent} (LND) if for any $f \in A$ there exists a number $n \in \NN$ such that $\delta^n(f) = 0$. The exponential map defines a bijection between locally nilpotent derivations on an algebra $A$ and rational $\GG_a$-actions on~$A$, see~\cite[Section~1.5]{Fr2006}. 
More precisely, an LND~$\delta$ on~$A$ defines the $\GG_a$-action $\alpha \cdot f = \exp(\alpha\delta)(f)$ for $\alpha \in \GG_a$, $f \in A$; conversely, given an action $\xi\colon \GG_a \times A \to A$, $(\alpha, f) \mapsto \alpha \cdot f$, we can recover the LND~$\delta$ as $\left(\frac{d\xi}{d\alpha}\right)|_{\alpha = 0}$. For an affine algebraic variety~$X$, $\GG_a$-actions on~$X$ are in bijection with $\GG_a$-actions on~$\KK[X]$, and so with LNDs on the algebra $\KK[X]$. 

Let $A = \bigoplus\limits_{u \in S} A_u$ be graded by a semigroup $S$. A derivation $\delta\colon A \to A$ is called \emph{homogeneous} if it maps homogeneous elements to homogeneous ones. It follows from the Leibniz rule that, for a domain $A$, a homogeneous derivation has the \emph{degree} $\deg \delta \in \ZZ S$ such that $\delta(A_u) \subseteq A_{u + \deg \delta}$ for any $u \in S$. 

Let $X_\sigma$ be an affine toric variety, and let $p_i \in N, \; 1 \le i \le m$, be primitive vectors on the rays of the cone~$\sigma$. For any $1 \le i \le m$, denote
\[
\mathfrak{R}_i = \{e \in M \mid \langle p_i, e\rangle = -1, \; \langle p_j, e\rangle \ge 0 \; \text{ for all } j \ne i, \; 1 \le j \le m\}.
\]
The elements of the set $\mathfrak{R} = \!\!\!\bigsqcup\limits_{1 \le i \le m} \!\!\!\mathfrak{R}_i$ are called the \emph{Demazure roots} of the toric variety~$X_\sigma$. It is easy to see that if $e \in \mathfrak{R}_i$, $\gamma$ is a face of~$\sigma$, and $e \in \gamma^\perp$, then the cone generated by $\gamma$ and $p_i$ is a face of~$\sigma$ as well. 

Demazure roots of $X_\sigma$ are in one-to-one correspondence with nonzero homogeneous LNDs on the algebra~$\KK[X_\sigma]$ with respect to grading~\eqref{torgrad_eq} up to proportionality. Homogeneous LNDs on~$\KK[X_\sigma]$ are, in turn, in one-to-one correspondence with $\GG_a$-actions on~$X_\sigma$ normalized by the acting torus~$\TT$, see~\cite[Theorem~2.7]{Li2010}. A Demazure root~$e$ is the degree of the corresponding LND, and the action of the torus $\TT$ on $\GG_a$ by conjugation is the multiplication by~$\chi^e$. 

\section{Classification results}

Let $X$ be an affine monoid with multiplication $X \times X \to X$. The dual homomorphism of algebras of regular functions is the \emph{comultiplication} $\KK[X] \to \KK[X] \otimes \KK[X]$. The comultiplication determines the multiplication, so affine monoids can be described in these terms. 

\smallskip

Let us formulate the main result. 

\begin{theorem}
\label{classif_theor}
Let $X$ be an affine monoid of corank one. Then $X = X_\sigma$ is toric, and the comultiplication $\KK[X_\sigma] \to \KK[X_\sigma] \otimes \KK[X_\sigma]$ has the form
\begin{equation} \label{comult_X_eq}
\chi^u \mapsto \chi^u \otimes \chi^u \, (1\otimes \chi^{e_1} + \chi^{e_2}\otimes1)^{\langle p, u\rangle},
\end{equation}
where $\KK[X_\sigma] = \!\bigoplus\limits_{u \in S_\sigma}\! \KK\chi^u$ as in equation~\eqref{torgrad_eq}, $p$ is the primitive vector on a ray of the cone~$\sigma$, and $e_1, e_2$ are Demazure roots corresponding to~$p$. Conversely, for any affine toric variety~$X_\sigma$, any primitive vector $p$ on a ray of the cone~$\sigma$, and any Demazure roots $e_1, e_2$ corresponding to the same~$p$, formula~\eqref{comult_X_eq} defines a monoid structure of corank one on~$X_\sigma$. 

The group of invertible elements is the toric open subset $X_\rho \subseteq X_\sigma$ isomorphic to the group $G_\chi$ for $\chi = \chi^{e_2 - e_1}$. 
\end{theorem}

\begin{remark}
Consider an automorphism of the lattice~$N$ such that the corresponding automorphism of the vector space~$N_\QQ$ preserves the cone~$\sigma$. Let $e_k$ map to~$e_k'$ under the dual automorphism of the lattice~$M$, where $e_k$, $k=1,2$, are two Demazure roots corresponding to the same ray of~$\sigma$. Then two monoid structures on $X_\sigma$ given by $e_1, e_2$ and by $e_1', e_2'$ are isomorphic. 
\end{remark}

\begin{example}
\label{mult_dim3_example}
Let us find all monoid structures of rank 2 on the quadratic cone \[X = \{vw=zt\} \subseteq \AA^4;\] see~\cite[Example~5]{DzZa2021} for the commutative case. It is an affine toric variety given by the cone~$\sigma$ with $p_1 = (1, 0, 0)$, $p_2 = (0, 1, 0)$, $p_3 = (1, 0, 1)$, and $p_4 = (0, 1, 1)$, see Figure~\ref{dim3pict}. Indeed, the semigroup $S_\sigma = M \cap \,\sigma^\vee$ is generated by vectors $(1,0,0)$, $(0,1,0)$, $(0,0,1)$, and $(1,1,-1)$; for the functions $v=\chi^{(1,0,0)}$, $w=\chi^{(0,1,0)}$, $z=\chi^{(0,0,1)}$, and $t=\chi^{(1,1,-1)}$, we obtain $\KK[X] = \KK[v, w, z, t]$ with relation $vw = zt$. 
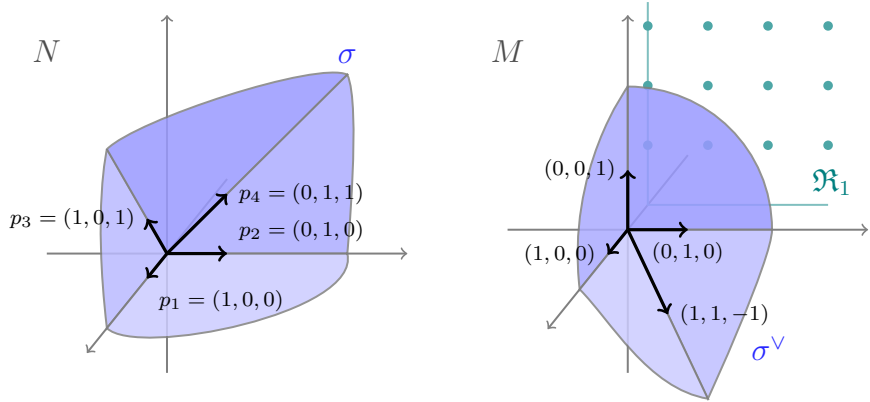
\begin{figure}[h]
\begin{center}

\tikzset{every picture/.style={line width=0.75pt}} 

\begin{tikzpicture}[x=0.75pt,y=0.75pt,yscale=-1,xscale=1]
    \tikzstyle{bluefill} = [fill=blue!20, draw = blue!80,opacity=0.85]
    \tikzstyle{violetfill} = [fill=violet!40, draw = violet!80,opacity=0.85]
    \tikzstyle{tealfill} = [fill=teal!15, draw = teal!80,opacity=0.85]

\coordinate (e1) at (-20,25);
\coordinate (e2) at (60,0);
\coordinate (e3) at (0,-60);

\coordinate (O) at (20,230);
\coordinate (Oxmax) at ($(O)+2*(e1)$);
\coordinate (Oymax) at ($(O)+2*(e2)$);
\coordinate (Ozmax) at ($(O)+2*(e3)$);
\coordinate (Oxmin) at ($(O)-1.5*(e1)$);
\coordinate (Oymin) at ($(O)-1*(e2)$);
\coordinate (Ozmin) at ($(O)-1*(e3)$);
\coordinate (A) at ($(O)+1.5*(e1)$);
\coordinate (B) at ($(O)+1.5*(e1)+1.5*(e3)$);
\coordinate (C) at ($(O)+1.5*(e2)+1.5*(e3)$);
\coordinate (D) at ($(O)+1.5*(e2)$);
\node[color=black!70] at ($(O)-1*(e2)+1.7*(e3)$) {$N$};

\draw [->,color=black!50] (Oxmin) -- (Oxmax); 
\draw [->,color=black!50] (Oymin) -- (Oymax); 
\draw [->,color=black!50] (Ozmin) -- (Ozmax); 

\filldraw [fill=blue!27,draw=black!50,opacity=0.85] (O) -- (A)  .. controls +(-7,-60) and +(0,0) .. (B) -- cycle;
\filldraw [fill=blue!40,draw=black!50,opacity=0.85] (O) -- (B)  .. controls +(15,-15) and +(-15,-7) .. (C) -- cycle;
\filldraw [fill=blue!32,draw=black!50,opacity=0.85] (O) -- (C) node[above,color=blue] {$\sigma$} .. controls +(7,15) and +(0,0) .. (D) -- cycle;
\filldraw [fill=blue!20,draw=black!50,opacity=0.85] (O) -- (D)  .. controls +(7,30) and +(15,15) .. (A) -- cycle;

\draw [->,color=black,very thick] (O) -- ($(O)+0.5*(e1)$) node[below right] {\tiny $p_1=(1,0,0)$};
\draw [->,color=black,very thick] (O) -- ($(O)+0.5*(e2)$) node[above right] {\tiny $p_2=(0,1,0)$};
\draw [->,color=black,very thick] (O) -- ($(O)+0.5*(e1)+0.5*(e3)$) node[left] {\tiny $p_3=(1,0,1)$};
\draw [->,color=black,very thick] (O) -- ($(O)+0.5*(e2)+0.5*(e3)$) node[right] {\tiny $p_4=(0,1,1)$};


\coordinate (O) at (250,230);
\coordinate (O) at ($(O)+0.2*(e3)$);
\coordinate (Oxmax) at ($(O)+2*(e1)$);
\coordinate (Oymax) at ($(O)+2*(e2)$);
\coordinate (Ozmax) at ($(O)+1.8*(e3)$);
\coordinate (Oxmin) at ($(O)-1.5*(e1)$);
\coordinate (Oymin) at ($(O)-1*(e2)$);
\coordinate (Ozmin) at ($(O)-1.2*(e3)$);
\coordinate (A) at ($(O)+1.2*(e1)$);
\coordinate (B) at ($(O)+1.2*(e3)$);
\coordinate (C) at ($(O)+1.2*(e2)$);
\coordinate (D) at ($(O)+1*(e1)+1*(e2)-1*(e3)$);
\node[color=black!70] at ($(O)-1*(e2)+1.5*(e3)$) {$M$};

\draw [->,color=black!50] (Oxmin) -- (Oxmax); 
\draw [->,color=black!50] (Oymin) -- (Oymax); 
\draw [->,color=black!50] (Ozmin) -- (Ozmax); 

\draw [color=teal!50] ($(O)-0.5*(e1)+1.7*(e3)$) -- ($(O)-0.5*(e1)$) -- ($(O)-0.5*(e1)+1.5*(e2)$);

\foreach \x in {0,...,3}
  \foreach \y in {1,2,3}
  {
  \node[draw=teal!70,circle,inner sep=1pt,fill=teal!70] at ($(O)+\x*0.5*(e2)+\y*0.5*(e3)-0.5*(e1)$) {};
  }
\node[color=teal] at ($(O)+1.7*(e2)+0.4*(e3)$) {$\mathfrak{R}_1$};

\filldraw [fill=blue!33,draw=black!50,opacity=0.85] (O) -- (A)  .. controls +(-7,-60) and +(0,0) .. (B) -- cycle;
\filldraw [fill=blue!40,draw=black!50,opacity=0.85] (O) -- (B)  .. controls +(40,0) and +(0,-40) .. (C) -- cycle;
\filldraw [fill=blue!25,draw=black!50,opacity=0.85] (O) -- (C) .. controls +(0,15) and +(0,0) .. node[midway,color=blue,below right] {$\sigma^\vee$} (D) -- cycle;
\filldraw [fill=blue!20,draw=black!50,opacity=0.85] (O) -- (D)  .. controls +(-30,-5) and +(15,15) .. (A) -- cycle;

\draw [->,color=black,very thick] (O) -- ($(O)+0.5*(e1)$) node[left] {\tiny $(1,0,0)$};
\draw [->,color=black,very thick] (O) -- ($(O)+0.5*(e2)$) node[below] {\tiny $(0,1,0)$};
\draw [->,color=black,very thick] (O) -- ($(O)+0.5*(e3)$) node[left] {\tiny $(0,0,1)$};
\draw [->,color=black,very thick] (O) -- ($(O)+0.5*(e1)+0.5*(e2)-0.5*(e3)$) node[right] {\tiny $(1,1,-1)$};

\end{tikzpicture}
\end{center}
\caption{The cones of $X$ in Example~\ref{mult_dim3_example}.}
\label{dim3pict}
\end{figure}
All vectors $p_i$ are equivalent up to an automorphism of the lattice $N$, let $p = p_1$. The set of corresponding Demazure roots equals \[\mathfrak{R}_1 = \{(-1, k, l) \mid k \in \Zgezero, l \in \ZZ_{>0}\},\] take $e_1 = (-1, k_1, l_1)$ and $e_2 = (-1, k_2, l_2)$. According to Theorem~\ref{classif_theor}, the product of $x, y \in X$ is defined by the formula 
\[\chi^u(x*y) = \chi^u(x) \chi^u(y) \, (\chi^{e_1}(y) + \chi^{e_2}(x))^{\langle p, u\rangle},\]
i.e.
\begin{multline*}
(v_x, w_x, z_x, t_x)*(v_y, w_y, z_y, t_y) = \\ = (v_xw_y^{k_1}z_y^{l_1} + v_yw_x^{k_2}z_x^{l_2}, \, w_xw_y, \, z_xz_y, \, t_xw_y^{k_1+1}z_y^{l_1-1} + t_yw_x^{k_2+1}z_x^{l_2-1}).
\end{multline*}
\end{example}

\medskip

The proof of Theorem~\ref{classif_theor} is given in the next section. Let us specialize the result to the case of affine space. 

\begin{theorem}
\label{affspace_theor}
Any monoid structure of corank one on $\AA^n$ up to a polynomial change of variables is given by
\[(x_1, \ldots, x_n) * (y_1, \ldots, y_n) = (x_1y_1, \, \ldots, \, x_{n-1}y_{n-1}, \, x_1^{a_1}\ldots x_{n-1}^{a_{n-1}}y_n + y_1^{b_1}\ldots y_{n-1}^{b_{n-1}}x_n),\]
where $a_1, \ldots, a_{n-1}, b_1, \ldots, b_{n-1} \in \ZZ_{\ge0}$.
\end{theorem}

\begin{proof}
According to~\cite{BB-1} one can assume that the toric structure on $\AA^n$ from Theorem~\ref{classif_theor} is given by diagonal matrices. The cones $\sigma \subseteq N_\QQ$ and $\sigma^\vee \subseteq M_\QQ$ of the affine space~$\AA^n$ are the positive octants in the corresponding rational vector spaces. Since all rays of $\sigma$ are equivalent up to automorphism, we can consider only one ray~$\rho$ with the primitive vector $p = (0, \ldots, 0, 1) \in N$. It remains to choose two Demazure roots $e_1 = (b_1, \ldots, b_{n-1}, -1)$, $e_2 = (a_1, \ldots, a_{n-1}, -1)$, $a_1, \ldots, a_{n-1}, b_1, \ldots, b_{n-1} \in \ZZ_{\ge0}$, and use formula~\eqref{comult_X_eq} for all basis vectors $u=u_i$, corresponding to coordinates $x_i$ on~$\AA^n$, $1 \le i \le n$. For $1 \le i \le n-1$, we have $\chi^{u_i}(x*y) = \chi^{u_i}(x)\chi^{u_i}(y) = x_iy_i$, and for $i = n$ we obtain \[\chi^{u_n}(x*y) = \chi^{u_n}(x)\chi^{u_n}(y)(\chi^{e_1}(y)+\chi^{e_2}(x)) = x_ny_n(y_1^{b_1}\ldots y_{n-1}^{b_{n-1}}y_n^{-1} + x_1^{a_1}\ldots x_{n-1}^{a_{n-1}}x_n^{-1}).\]
\end{proof}

\section{Proof of Theorem~\ref{classif_theor}}
\label{proofclass_sec}

The proof is divided into four subsections below. The first three subsections contain the proof of the direct implication of the theorem, and the last subsection provides the converse implication. 

\subsection{Preliminary results}

Let $X$ be an affine monoid of corank one. 
By Proposition~\ref{toric_theor}, the variety $X$ is toric with respect to the torus~$\TT$, and the open subset $G_\chi \subseteq X$ is invariant with respect to~$\TT$, i.e. is a toric variety as well. Let $X = X_\sigma$ for a cone $\sigma \subseteq N_\QQ$ and $G_\chi = X_\rho$ for a cone $\rho \subseteq N_\QQ$. 
As an algebraic variety, the group $G_\chi = \GG_a \leftthreetimes T$ is isomorphic to the direct product of the line $\GG_a$ and the torus $T$. 
Then the cone $\rho \subseteq N_\QQ$ is a ray and the dual cone $\rho^\vee \subseteq M_\QQ$ is a halfspace. 

Notice that the inclusion $G_\chi = X_\rho \hookrightarrow X = X_\sigma$ restricts to the identity map on the open orbit of the torus action. So it comes from the inclusion of semigroups $S_\sigma \subseteq S_\rho$ and of cones $\rho \hookrightarrow \sigma$ as in subsection~\ref{Urho_subsec}. Since $X_\rho \subseteq X_\sigma$ is an open embedding, the cone $\rho$ is a face of the cone $\sigma$, see subsection~\ref{Urho_subsec}. 
So $\rho$ is a ray of the cone~$\sigma$. 

The action of $\TT = \Im\theta$ on~$G_\chi$ comes from left and right multiplications of $T$ on~$G_\chi$, see Proposition~\ref{toric_theor}. If we act on a point $(\alpha, t] \in G_\chi$ by an element $\theta(t_1, t_2) \in \TT$, $t_1, t_2 \in T$, we obtain a point
\begin{equation} \label{TactonGchi_eq}
\theta(t_1, t_2)\bigl((\alpha, t]\bigr) = (0, t_1](\alpha, t](0, t_2^{-1}] = (\chi(t_1)\alpha, \, t_1tt_2^{-1}].
\end{equation}
Notice that $\chi(t_1) = \chi(t_1')$ does not depend on a choice of a representative $\theta(t_1, t_2) = \theta(t_1', t_2') \in \TT$ since $\Ker \theta = \diag \Ker \chi$. So $\theta(t_1, t_2) \mapsto \chi(t_1)$ is a character of~$\TT$. It follows that the coordinate function $\falpha \in \KK[G_\chi]$ is homogeneous with respect to grading~\eqref{torgrad_eq} corresponding to the action of~$\TT$. The function $\fbeta = \chi^{-1}(t)\falpha \in \KK[G_\chi]$ is homogeneous as well. 
Multiplying $\falpha$ and $\fbeta$ by an appropriate scalar we may assume that $\falpha = \chi^a$ and $\fbeta = \chi^b$ for some $a, b \in S_\rho$. Notice that $\falpha \in \KK[G_\chi]$ and $1/\falpha \notin \KK[G_\chi]$, whence $a \in S_\rho$ does not belong to the face $\rho^\perp$, i.e. $\langle p, a\rangle > 0$. 

\subsection{Comultiplication on $\KK[G_\chi]$} Let $G$ be the group of invertible elements in $X$. Then $G$ is open in $X$, the algebra $\KK[X]$ is a subalgebra of $\KK[G]$, and the comultiplication on $\KK[X]$ is the restriction of the comultiplication~$\KK[G] \to \KK[G]\otimes\KK[G]$ corresponding to the multiplication in the group $G$. 

\medskip

Therefore, we are interested in the comultiplication on $\KK[G_\chi]$. 

\begin{lemma}
\label{comult_Gchi_lem}
The comultiplication $\KK[G_\chi] \to \KK[G_\chi] \otimes \KK[G_\chi]$ is given by 
\begin{equation} \label{comult_Gchi_eq2}
\chi^u \mapsto \chi^u \otimes \chi^u \, (1 \otimes \falpha^{-1} + \fbeta^{-1} \otimes 1)^{\langle p, u\rangle},
\end{equation}
where $u \in S_\rho$ and $\KK[G_\chi] = \!\bigoplus\limits_{u \in S_\rho} \!\KK\chi^u$. 
\end{lemma}

\begin{proof}
Consider the orbits of the action of the torus $\TT$ on $X_\rho = G_\chi = \GG_a \leftthreetimes T$. The cone~$\rho$ has two faces $0$ and $\rho$, so according to subsection~\ref{Torb_subsec} there are two orbits $X_\rho = O_0 \cup O_\rho$. By formula~\eqref{TactonGchi_eq}, we see that $O_0 = \{\falpha \ne 0\}$ is the open orbit and $O_\rho = \{\falpha = 0\} = T$ is its complement. 
The torus $\Hom(M(\rho), \KK^\times)$ for the sublattice $M(\rho) = \rho^\perp \cap M$ in~$M$ is identified with $O_\rho$, so $\KK[T] = \KK[M(\rho)]$. 

Since $\KK[G_\chi] = 
\KK[T][\falpha]$, the semigroup~$S_\rho$ is generated by the lattice~$M(\rho)$ and $a\in S_\rho$, where $\falpha = \chi^a$. 
Notice that $\langle p, u_1+u_2\rangle = \langle p, u_1\rangle+\langle p, u_2\rangle$ and $\chi^{u_1+u_2} = \chi^{u_1}\chi^{u_2}$. Then it is sufficient to check formula~\eqref{comult_Gchi_eq2} for $u \in M(\rho)$ and for $u = a$. 

For $u \in M(\rho)$ we obtain $\chi^u \mapsto \chi^u \otimes \chi^u$ since $\chi^u \in \KK[T]$ are multiplied as characters in $T \subseteq G_\chi$. 
Notice that for $u \in M(\rho)$ formula~\eqref{comult_Gchi_eq2} also turns into $\chi^u \mapsto \chi^u \otimes \chi^u$ since $\langle p, u\rangle = 0$. 

By equation~\eqref{Gchi_mult_eq} we also have
$\falpha \mapsto \falpha\otimes1 + \chi(t)\otimes\falpha$.
Let us check that it coincides with formula~\eqref{comult_Gchi_eq2} for $\chi^u = \falpha$, i.e. for $u=a$. Notice that \[\falpha\otimes1 + \chi(t)\otimes\falpha = \falpha\otimes\falpha \, (1\otimes\falpha^{-1} + \chi(t)\falpha^{-1} \otimes 1),\] 
where $\chi(t)\falpha^{-1} = \fbeta^{-1}$. The exponent $\langle p,a\rangle$ is equal to~$1$ since the sublattice $M(\rho) = \{u \in M \mid \langle p, u\rangle = 0\}$ and the element $a \in \rho^\vee$ generate a semigroup $S_\rho$ in $M$ such that $\ZZ S_\rho = M$. 
\end{proof}

\subsection{Derivations and Demazure roots}
By Lemma~\ref{comult_Gchi_lem}, the comultiplication on $\KK[X_\sigma]$ is the restriction of the comultiplication on $\KK[G_\chi]$, see equation~\eqref{comult_Gchi_eq2}:
\[\chi^u \mapsto \chi^u \otimes \chi^u \, (1 \otimes \falpha^{-1} + \fbeta^{-1} \otimes 1)^{\langle p, u\rangle}.\]
Denote $e_1 = -a$ and $e_2 = -b$, where $\falpha = \chi^a$ and $\fbeta = \chi^b$, $a, b \in S_\rho$. We have seen in the proof of Lemma~\ref{comult_Gchi_lem} that $\langle p, a\rangle = 1$, so $\langle p, e_1\rangle = -1$. In the same way $\langle p, e_2\rangle = -1$. 

Consider two $\GG_a$-actions on~$G_\chi$ coming from left and right multiplication by $\GG_a \subseteq G_\chi$: 
$$\begin{aligned}
\GG_a \times G_\chi \to G_\chi, &\quad \alpha' \cdot (\alpha, \, t] = (\alpha' + \alpha, \, t];\\
G_\chi \times \GG_a \to G_\chi, &\quad [t, \beta) \cdot \beta' = [t, \beta + \beta').
\end{aligned}$$
To obtain the corresponding LNDs on $\KK[G_\chi] = \KK[T][\falpha] = \KK[T][\fbeta]$, we differentiate dual actions of~$\GG_a$ on~$\KK[G_\chi]$ at $\alpha'=0$ and $\beta'=0$: 
\[\delta_l = \frac{d}{d\falpha}; \quad \delta_r = \frac{d}{d\fbeta}.\] 

Left and right multiplications on $G_\chi$ can be extended to $X_\sigma$. Since the algebra $\KK[X_\sigma]$ is invariant with respect to the action by left and right multiplication, the algebra $\KK[X_\sigma]$ is also invariant with respect to derivations~$\delta_l$ and~$\delta_r$. Notice that $\delta_l$ and~$\delta_r$ are homogeneous LNDs on $\KK[X]$ with respect to the grading $\KK[X_\sigma] = \!\bigoplus\limits_{u \in S_\sigma} \!\KK\chi^u$ with degrees $e_1$ and $e_2$, respectively, since $\falpha^{-1} = \chi^{e_1}$ and $\fbeta^{-1} = \chi^{e_2}$. It follows that $e_1$ and $e_2$ are Demazure roots of~$X_\sigma$. Since $\langle p, e_1\rangle = \langle p, e_2\rangle = -1$, they correspond to the ray~$\rho$. Thus we proved that the comultiplication on $\KK[X_\sigma]$ is given by formula~\eqref{comult_X_eq}. 

\subsection{The inverse implication}
\label{invproof_subsect}
First let us provide the proof of the following lemma. 
\begin{lemma} \label{comult_restr_lem}
Let $G$ be an affine algebraic group and $G \hookrightarrow X$ be an open embedding into an affine algebraic variety $X$. Let $\mu\colon \KK[G] \to \KK[G]\otimes\KK[G]$ be the comultiplication on $\KK[G]$ and $\mu(\KK[X]) \subseteq \KK[X] \otimes \KK[X]$. Then the restriction of~$\mu$ to~$\KK[X]$ defines a comultiplication on $\KK[X]$ that corresponds to some monoid structure on~$X$. 
\end{lemma}

\begin{proof}
If $\mu(\KK[X]) \subseteq \KK[X] \otimes \KK[X]$, then the dual map $X \times X \to X$ is a morphism, which restricts to the group multiplication $G \times G \to G$. Since the former is associative and admits a unity, the same holds for~$X \times X \to X$. 
\end{proof}

Now let $X_\sigma$ be an affine toric variety, $p$ be the primitive vector on a ray~$\rho$ of the cone $\sigma \subseteq N_\QQ$, and $e_1, e_2 \in M$ be Demazure roots corresponding to~$p$. First let us prove that formula~\eqref{comult_X_eq} defines a map to $\KK[X_\sigma] \otimes \KK[X_\sigma]$. 
We have
\begin{equation}
\label{comult_X_binom_eq}
\chi^u \mapsto \chi^u \otimes \chi^u \, (1 \otimes \chi^{e_1} + \chi^{e_2} \otimes 1)^{\langle p, u\rangle} = \sum\limits_{i+j=\langle p, u\rangle} \begin{pmatrix}\langle p, u\rangle\\i\end{pmatrix} \chi^{u+ie_2} \otimes \chi^{u+je_1}.
\end{equation}
Notice that, according to the definition of a Demazure root, $\langle p, u+ie_2\rangle = \langle p,u \rangle - i \ge 0$, and for the primitive vector $p'$ on any other ray of the cone~$\sigma$ we have $\langle p', u+ie_2\rangle = \langle p',u \rangle + i \langle p', e_2\rangle \ge 0 + 0 = 0$ as well. So $\chi^{u+ie_2} \in \KK[X_\sigma]$. Similarly, $\chi^{u+je_1} \in \KK[X_\sigma]$.

Consider the natural open embedding of the toric variety $X_\rho$ with the cone~$\rho$ into~$X_\sigma$, see subsection~\ref{Urho_subsec}. Let us show that the unity of the given multiplication is the point $x_\rho$, see subsection~\ref{Torb_subsec}. Recall that by the definition of this point for $u \in S_\rho$ we have $\chi^u(x_\rho) = 1$ if $u \in \rho^\perp$ and $\chi^u(x_\rho)=0$ otherwise. Suppose $x \in X_\sigma$, $u \in S_\sigma$. Then 
\[\chi^u(x * x_\rho) = \sum\limits_{i+j=\langle p, u\rangle} \begin{pmatrix}\langle p, u\rangle\\i\end{pmatrix} \chi^{u+ie_2}(x) \chi^{u+je_1}(x_\rho) = \chi^u(x).\]
Indeed, the last equality holds since $\langle p, u+je_1\rangle = 0$ if and only if $j = \langle p,u\rangle$. 

Let us prove that the set of invertible elements in the monoid $X_\sigma$ is a subset of~$X_\rho$. Consider a point $x \in X_\sigma$, and assume that there exists the inverse $y = x^{-1} \in X_\sigma$. Then for any $u \in S_\sigma$ in the relative interior of the facet $\rho^\perp \cap \sigma^\vee$ we have $\langle p,u\rangle = 0$, so formula~\eqref{comult_X_binom_eq} gives
\[
\chi^u(x*y) = \chi^u(x)\chi^u(y), 
\]
that has to equal~$1$ since \[\chi^u(x_\rho) = \begin{cases}1 \;\text{ if } u \in \rho^\perp,\\
0 \;\text{ otherwise}.\end{cases}\] So $\chi^u(x) \in \KK^\times$ if $u$ is in the relative interior of the facet $\rho^\perp \cap \sigma^\vee$. According to the description of the open subset~$X_\rho$ in subsection~\ref{Urho_subsec}, we obtain $x \in X_\rho$. 

Finally, since $e_1, e_2$ are Demazure roots with $\langle p, e_1\rangle = \langle p, e_2\rangle = -1$, it follows that $-e_1, -e_2 \in \rho^\vee$, so the characters $\falpha := \chi^{-e_1}$ and $\fbeta := \chi^{-e_2}$ are regular functions on~$X_\rho$. 
The restriction of comultiplication~\eqref{comult_X_eq} to the open subset~$X_\rho$ coincides with formula~\eqref{comult_Gchi_eq2}, so all points in $X_\rho$ are invertible and $X_\rho \cong G_\chi$, where $\chi = \falpha/\fbeta$. 
Thus $X_\sigma$ is a monoid with the group of invertible elements $X_\rho$ of corank one by Lemma~\ref{comult_restr_lem}. This completes the proof of Theorem~\ref{classif_theor}.


\section{Idempotent elements}

The classification of affine monoids of corank one from Theorem~\ref{classif_theor} allows us to study algebraic properties of such monoids. 

\medskip

Let us give the classification of \emph{idempotents} in $X$, i.e. elements $x \in X$ such that $x*x = x$. 

\begin{theorem}
\label{idemp_theor}
In the notation of Theorem~\ref{classif_theor}, for any face~$\gamma$ of the cone $\sigma$ the set of idempotents $E_\gamma$ in the orbit $O_\gamma$ in~$X$ is as follows:
\begin{enumerate}
  \item $E_\gamma = \{x_\gamma\}$ if $\rho$ is a ray of $\gamma$;
  \item $E_\gamma = \varnothing$ if $\rho$ is not a ray of $\gamma$ and $e_1, e_2 \notin \gamma^\perp$;
  \item $E_\gamma = \varnothing$ if $\rho$ is not a ray of $\gamma$ and $e_1, e_2 \in \gamma^\perp$;
  \item $E_\gamma = O_\gamma \cap \,\{\chi^u = 1 \;\; \forall u \in \cone(\gamma, \rho)^\perp \cap S_\sigma\}$ otherwise. 
\end{enumerate}
\end{theorem}

\begin{remark}
\label{def_Dem_cone_rem}
Let $e$ be a Demazure root corresponding to the ray~$\rho$ of the cone~$\sigma$, which is not a ray of the cone~$\gamma$, and $e \in \gamma^\perp$. Then $\cone(\gamma, \rho)$ is a face of~$\sigma$, see subsection~\ref{LND_Dem_subsec}. 
\end{remark}

A proof of Theorem~\ref{idemp_theor} is given at the end of this section. 

\begin{example}
\label{idempA2_example}
Let $X = \AA^2$ and the multiplication be given by formula
\[(x_1, x_2) * (y_1, y_2) = (x_1y_1, \, x_1^ay_2 + x_2), \quad a > 0.\]
Denote a basis of $N$ by $s_1, s_2$. According to the proof of Theorem~\ref{affspace_theor}, the monoid corresponds to the cone $\sigma = \cone(s_1, s_2)$, the ray $\rho$ with the primitive vector ${p = s_2 \in N}$ and Demazure roots $e_1 = (a,-1) \in M$, $e_2 = (0,-1) \in M$, see Figure~\ref{A2pict}. 

\begin{figure}[h]
\begin{center}
\tikzset{every picture/.style={line width=0.75pt}} 

\begin{tikzpicture}[x=0.75pt,y=0.75pt,yscale=-1,xscale=1]
    \tikzstyle{conefill} = [fill=blue!20, draw = black!40]
    \tikzstyle{conefill_gamma} = [fill=red!20, draw = black!70]

\coordinate (e1) at (20,0);
\coordinate (e2) at (0,-20);


\coordinate (O) at (120,230);
\coordinate (Oxmax) at ($(O)+5*(e1)$);
\coordinate (Oymax) at ($(O)+5*(e2)$);
\coordinate (Oxmin) at ($(O)-(e1)$);
\coordinate (Oymin) at ($(O)-2*(e2)$);
\coordinate (A) at ($(O)+4.5*(e1)$);
\coordinate (B) at ($(O)+4.5*(e2)$);
\coordinate (T) at ($(O)+4*(e2)-2*(e1)$);

\filldraw [conefill] 
(O) node[below left] {$0$} .. controls (O) and (B) .. node[very near end,below left] {$\rho$} (B)  .. node [midway,below left,blue] {$\sigma$} controls +(60,-10) and +(10,-60) .. (A) .. node[very near start,below left] {$\gamma$} controls (A) and (O) .. (O) -- cycle;

\draw [->,color=black!50] (Oxmin) -- (Oxmax);
\draw [->,color=black!50] (Oymin) -- (Oymax);

\draw[->,very thick] (O) -- ($(O)+(e1)$) node [below] {\small $s_1$};
\draw[->,very thick] (O) -- ($(O)+(e2)$) node [left] {\small $p=s_2$};
\draw[black!60] (T) node {$N$};


\coordinate (O) at ($(O)+10*(e1)$);
\coordinate (Oxmax) at ($(O)+5*(e1)$);
\coordinate (Oymax) at ($(O)+5*(e2)$);
\coordinate (Oxmin) at ($(O)-(e1)$);
\coordinate (Oymin) at ($(O)-2*(e2)$);
\coordinate (A) at ($(O)+4.5*(e1)$);
\coordinate (B) at ($(O)+4.5*(e2)$);
\coordinate (T) at ($(O)+4*(e2)-2*(e1)$);

\filldraw [conefill] 
(O) node[below left] {$0$} .. controls (O) and (B) .. node[very near end,below left] {$\gamma^\perp$} (B)  .. node [midway,below,blue] {$\sigma^\vee$} controls +(60,-10) and +(10,-60) .. (A) .. node[very near start,above left] {$\rho^\perp$} controls (A) and (O) .. (O) -- cycle;

\draw [->,color=black!50] (Oxmin) -- (Oxmax);
\draw [->,color=black!50] (Oymin) -- (Oymax);

\draw[->,very thick] (O) -- ($(O)+(e1)$) node [below] {\small $u_1$};
\draw[->,very thick] (O) -- ($(O)+(e2)$) node [left] {\small $u_2$};
\draw[black!60] (T) node {$M$};
\node[draw,circle,teal,inner sep=1pt,fill] at ($(O)+3*(e1)-(e2)$) {};
\node[teal,right] at ($(O)+3*(e1)-(e2)$) {\scriptsize $e_1=(a,-1)$};
\node[draw,circle,teal,inner sep=1pt,fill] at ($(O)-(e2)$) {};
\node[teal,below right] at ($(O)-(e2)$) {\scriptsize $e_2=(0,-1)$};


\coordinate (O) at ($(O)+10*(e1)$);
\coordinate (Oxmax) at ($(O)+5*(e1)$);
\coordinate (Oymax) at ($(O)+5*(e2)$);
\coordinate (Oxmin) at ($(O)-2*(e1)$);
\coordinate (Oymin) at ($(O)-2*(e2)$);
\coordinate (A) at ($(O)+4.5*(e1)$);
\coordinate (B) at ($(O)+4.5*(e2)$);
\coordinate (T) at ($(O)+4*(e2)-(e1)$);

\draw [->,color=black!50] (Oxmin) -- (Oxmax);
\draw [->,color=black!50] (Oymin) -- (Oymax);

\draw[very thick,red] (Oymin) -- (B) node [below right] {$E_\gamma=O_\gamma$};
\draw[very thick,teal] (Oxmin) -- (A) node [above] {$O_\rho$};
\node[draw,circle,blue,inner sep=1pt,fill] at (O) {};
\node[draw,circle,red,inner sep=1pt,fill] at ($(O)+1.5*(e1)$) {};
\node[red,below right] at ($(O)+(e1)$) {\scriptsize $(1,0) \in E_\rho$};
\draw (O) node[above right,blue] {\scriptsize $E_\sigma=O_\sigma$} node[below left,blue] {\scriptsize $(0,0)$};
\draw ($(O)+3*(e1)+2.5*(e2)$) node {$O_0$};
\draw[black!60] (T) node {$\AA^2$};


\end{tikzpicture}
\end{center}
\caption{The cones, $\TT$-orbits, and idempotents in $X$ in Example~\ref{idempA2_example}.}
\label{A2pict}
\end{figure}
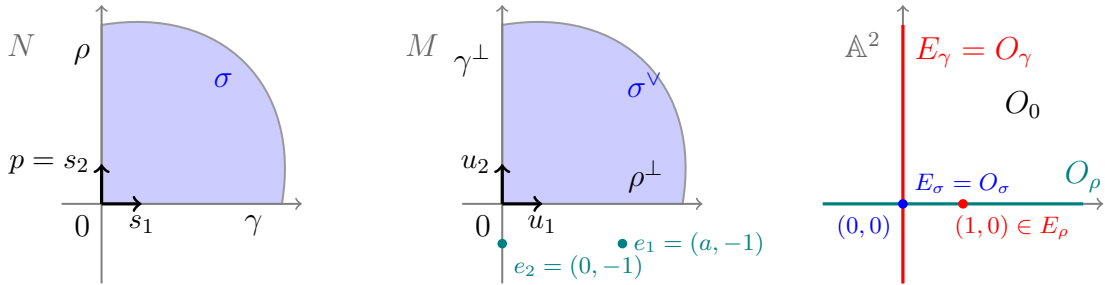

The cone $\sigma$ has four faces $0, \,\rho, \,\gamma = \cone(s_1)$, and $\sigma$. According to Theorem~\ref{idemp_theor}, idempotents in~$X$ belong to toric orbits $O_\rho$, $O_\sigma$ in item~(a), and $O_\gamma$ in item~(d). In $O_\rho$ and~$O_\sigma$ we obtain idempotents $x_\rho = (1,0)$ and $x_\sigma = (0,0)$. The set $E_\gamma$ of idempotents in $O_\gamma$ is given by equations $\chi^u = 1$ for $u \in \cone(\gamma, \rho)^\perp = 0$, whence $E_\gamma=O_\gamma$. So the set of idempotents consists of the line $(0,x_2)$, $x_2 \in \KK$, and the point $(1,0)$. It coincides with the solutions of $(x_1^2, \, x_1^ax_2 + x_2) = (x_1, x_2)$. 
\end{example}

\begin{example}
\label{idempA4_example}
Let $X = \AA^4$ and the multiplication rule be as follows: 
\begin{equation}
\label{mult_A4ex_eq}
(x_1, x_2, x_3, x_4) * (y_1, y_2, y_3, y_4) = (x_1y_1, \, x_2y_2, \, x_3y_3, \, x_3^ay_4 + y_2^by_3^cx_4), \quad a,b,c > 0.
\end{equation}
Denote a basis of $N$ by $s_i$, $1 \le i \le 4$. According to the proof of Theorem~\ref{affspace_theor}, the monoid corresponds to the cone $\sigma = \cone(s_1, s_2,s_3,s_4)$, the ray $\rho$ with the primitive vector ${p = s_4 \in N}$, and Demazure roots $e_1 = (0,b,c,-1) \in M$, $e_2 = (0,0,a,-1) \in M$. 

For any subset $I \subseteq \{1,2,3,4\}$, consider the face $\gamma = \gamma(I) = \cone\limits_{i \in I} s_i$ of the cone~$\sigma$, which corresponds to the torus orbit
\[O_{\gamma} = \{(x_1, x_2, x_3, x_4) \mid x_i = 0 \text{ if } i \in I \text{\; and \;} x_i \ne 0 \text{ if } i \notin I\}.\]
The point $x_{\gamma} \in O_{\gamma}$ is the point with $i$-th coordinate $0$ if $i \in I$ and $1$ if $i \notin I$. Notice that $e_1 \in \gamma^\perp$ if and only if $I \subseteq \{1\}$, and $e_2 \in \gamma^\perp$ if and only if $I \subseteq \{1, 2\}$. 

Let us find the set of idempotents in an orbit $O_\gamma$, $\gamma = \gamma(I)$, comparing results of Theorem~\ref{idemp_theor} and direct computations. By equation~\eqref{mult_A4ex_eq}, a point $(x_1, x_2, x_3, x_4) \in \AA^4$ is an idempotent if
\begin{equation}
\label{mult_A4ex_idemp_eq}
(x_1^2, x_2^2, x_3^2, x_3^ax_4 + x_2^bx_3^cx_4) = (x_1, x_2, x_3, x_4).
\end{equation}

(a) Let $\rho$ be a ray of $\gamma$, i.e. $4 \in I$ (there are eight such cones $\gamma$). Then by Theorem~\ref{idemp_theor}, there exists exactly one idempotent $x_\gamma = (x_1, x_2, x_3, 0)$ in~$O_\gamma$, where $x_i = 0$ if $i \in I$ and $x_i = 1$ otherwise. The same follows from equation~\eqref{mult_A4ex_idemp_eq} under the assumption $x_4 = 0$. 

(b) Let $\rho$ be not a ray of $\gamma$ and $e_1, e_2 \notin \gamma^\perp$, i.e. $4 \notin I$, $3 \in I$. Then there are no idempotents in $O_\gamma$ by Theorem~\ref{idemp_theor}. This agrees with formula~\eqref{mult_A4ex_idemp_eq}: from $x_4 \ne 0, x_3 = 0$ it follows that $x_3^ax_4 + x_2^bx_3^cx_4 = 0 \ne x_4$. 

(c) Let $\rho$ be not a ray of $\gamma$ and $e_1, e_2 \in \gamma^\perp$, i.e. $I \subseteq \{1\}$. There are no idempotents in such $O_\gamma$ by Theorem~\ref{idemp_theor} as well. Also, from equation~\eqref{mult_A4ex_idemp_eq} and $x_2, x_3, x_4 \ne 0$ it follows that $x_2 = x_3 = 1, \, x_4 + x_4 = x_4, \, x_4 \ne 0$, a contradiction. 

(d) Let $\rho$ be not a ray of $\gamma$ and $e_1 \notin \gamma^\perp, e_2 \in \gamma^\perp$, i.e. $I \subseteq \{1, 2\}$ and $2 \in I$. Consider two possible cases.

(d1) Let $I = \{1,2\}$, i.e. $O_\gamma = \{(0,0,x_3,x_4) \mid x_3, x_4 \in \KK^\times\}$. Denote by $u_i, \, 1 \le i \le 4$, the basis of~$M$ dual to $s_i, \, 1 \le i \le 4$. The semigroup of integer points in the face \[\cone(\gamma, \rho)^\perp \cap \sigma^\vee = \cone(s_1, s_2, s_4)^\perp \cap \cone(u_1, u_2, u_3, u_4) = \cone(u_3)\] of the dual cone $\sigma^\vee$ is generated by~$u_3$, so by Theorem~\ref{idemp_theor} the set of idempotents in $O_\gamma$ equals
\[E_\gamma = O_\gamma \cap \,\{\chi^u = 1 \;\; \forall u \in \cone(\gamma,\rho)^\perp \cap S_\sigma\} = \{(0,0,1,x_4) \mid x_4 \in \KK^\times\}.\]
This agrees with equation~\eqref{mult_A4ex_idemp_eq}, which is equivalent to $x_3 = 1$ under the assumptions $x_1 = x_2 = 0$, $x_3, x_4 \ne 0$. 

(d2) Let $I = \{2\}$. In the same way we obtain
\[\cone(\gamma, \rho)^\perp \cap \sigma^\vee = \cone(u_1, u_3)\]
and 
\[E_\gamma = O_\gamma \cap \,\{\chi^u = 1 \;\; \forall u \in \gamma^\perp \cap \rho^\perp \cap S_\sigma\} = \{(1,0,1,x_4) \mid x_4 \in \KK^\times\}.\]

Gathering all cases together, we obtain that the set of idempotents consists of two lines $(0,0,1,x_4)$ and $(1,0,1,x_4)$, $x_4 \in \KK$, and six points 
\[(0, 0, 0, 0), (0, 1, 0, 0), (0, 1, 1, 0), (1, 0, 0, 0), (1, 1, 0, 0), (1, 1, 1, 0).\]
\end{example}

It is known that an irreducible commutative algebraic monoid in characteristic zero has a finite number of idempotents~\cite[Section~3.5.3, Exercise 16b]{Re2005}. This agrees with our result. 

\begin{corollary}
\label{idemp_comm_cor}
In the notation of Theorem~\ref{classif_theor}, the number of idempotents in $X$ is finite if and only if $e_1, e_2$ belong to $\gamma^\perp$ simultaneously, where $\gamma$ ranges over faces of~$\sigma$ such that $\rho$ is not a ray of~$\gamma$. In such a case, the number of idempotents equals the number of faces of~$\sigma$ that contain~$\rho$. In particular, this holds if $X$ is commutative. 
\end{corollary}

\begin{proof}
The number of idempotents is finite if and only if case $(d)$ is impossible. If $X$ is commutative, then $e_1 = e_2$. 
\end{proof}

\begin{proof}[Proof of Theorem~\ref{idemp_theor}]
Let $x \in O_\gamma$ be an idempotent in~$X_\sigma$. From now on $u \in S_\sigma$. First recall that, according to~\eqref{comult_X_binom_eq}, for any $x_1, x_2 \in X_\sigma$ the multiplication is given by
\begin{equation}
\label{comult_X_binom_eq2}
\chi^u(x_1 * x_2) = \sum\limits_{i+j=\langle p, u\rangle} \begin{pmatrix}\langle p, u\rangle\\i\end{pmatrix} \chi^{u+ie_2}(x_1) \chi^{u+je_1}(x_2).
\end{equation}
Informally, $u+ie_1$ and $u+ie_2$, $0 \le i \le \langle p,u\rangle$, form two segments in $S_\sigma$ between~$u$ and the facet $\rho^\perp \cap \sigma^\vee$, see Figure~\ref{4cases_pict} below. As $i$ increases, the element $u + ie_k$, $k = 1,2$, becomes closer to the facet $\rho^\perp \cap \sigma^\vee$ and not closer to all other facets by the definition of a Demazure root, see subsection~\ref{invproof_subsect}. In particular, $u + ie_k \in \rho^\perp$ if and only if $i = \langle p,u\rangle$, where $k = 1,2$. 

First consider $u \in \rho^\perp$, see Figure~\ref{4cases_pict}. Then the sum in~\eqref{comult_X_binom_eq2} has only one summand, so for the idempotent~$x$ we have $\chi^u(x) = \chi^u(x*x) = \chi^u(x)\chi^u(x)$.
Thus
\begin{equation}
\label{idemp_rhoperp_eq}
\chi^u(x) = 0,1 \quad \text{ for any } u \in \rho^\perp. 
\end{equation}

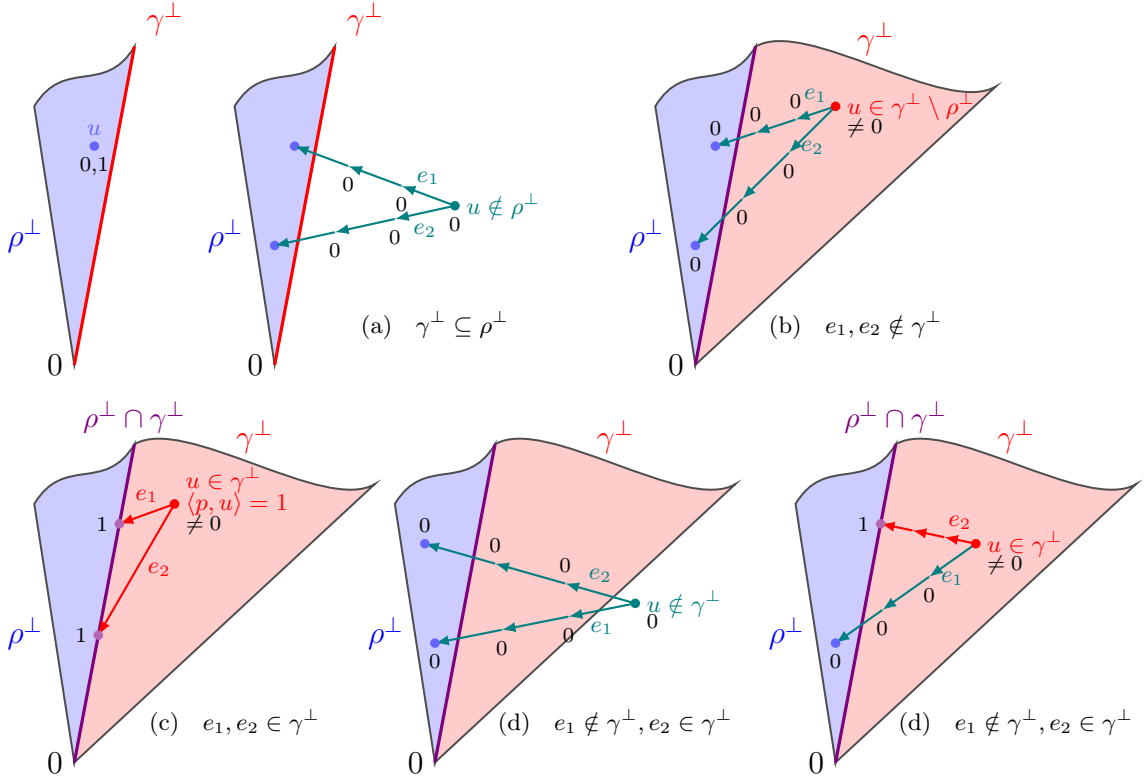
\begin{figure}[h]
\begin{center}
\tikzset{every picture/.style={line width=0.75pt}} 

\begin{tikzpicture}[x=0.75pt,y=0.75pt,yscale=-1,xscale=1]
    \tikzstyle{conefill} = [fill=blue!20, draw = black!70]
    \tikzstyle{conefill_gamma} = [fill=red!20, draw = black!70]

\coordinate (OA) at (-20,-130);
\coordinate (OB) at (30,-160);


\coordinate (O) at (120,230);
\coordinate (Ou) at (10,-110);
\coordinate (A) at ($(O)+(OA)$);
\coordinate (B) at ($(O)+(OB)$);

\filldraw [conefill] 
(O) node[left] {$0$} .. controls (O) and (B) .. (B) .. controls +(-15,25) and +(15,-25) .. (A) .. node [midway,left,blue] {$\rho^\perp$} controls (A) and (O) .. (O) -- cycle;
\draw[very thick,red] (O) -- (B) node [above right] {$\gamma^\perp$};

\coordinate [label=above:{\textcolor{blue!60}{\scriptsize $u$}},label=below:{\tiny 0,1}] (u) at ($(O)+(Ou)$);
\node[draw,circle,blue!60,inner sep=1pt,fill] at (u) {};

\coordinate (O) at (220,230);
\coordinate (Ou) at (90,-80);
\coordinate (Oend1) at (10,-110);
\coordinate (Oend2) at (0,-60);
\coordinate (A) at ($(O)+(OA)$);
\coordinate (B) at ($(O)+(OB)$);
\coordinate (end1) at ($(O)+(Oend1)$);
\coordinate (end2) at ($(O)+(Oend2)$);

\filldraw [conefill] 
(O) node[left] {$0$} .. controls (O) and (B) .. (B) .. controls +(-15,25) and +(15,-25) .. (A) .. node [midway,left,blue] {$\rho^\perp$} controls (A) and (O) .. (O) -- cycle;
\draw[very thick,red] (O) -- (B) node [above right] {$\gamma^\perp$};

\coordinate [label=right:{\textcolor{teal}{\scriptsize $u \notin \rho^\perp$}}] (u) at ($(O)+(Ou)$);
\node[draw,circle,teal,inner sep=1pt,fill] at (u) {};
\node[draw,circle,blue!60,inner sep=1pt,fill] at (end1) {};
\node[draw,circle,blue!60,inner sep=1pt,fill] at (end2) {};
\coordinate (e1) at ($1/3*(end1)-1/3*(u)$);
\coordinate (e2) at ($1/3*(end2)-1/3*(u)$);
\draw[thick,teal,-latex] (u) node [below] {\tiny\textcolor{black}{0}} -- node [midway,above] {\scriptsize $e_1$} ($(u)+(e1)$);
\draw[thick,teal,-latex] ($(u)+(e1)$) node [below] {\tiny\textcolor{black}{0}} -- ($(u)+2*(e1)$);
\draw[thick,teal,-latex] ($(u)+2*(e1)$) node [below] {\tiny\textcolor{black}{0}} -- (end1);
\draw[thick,teal,-latex] (u) -- node [midway,below] {\scriptsize $e_2$} ($(u)+(e2)$);
\draw[thick,teal,-latex] ($(u)+(e2)$) node [below] {\tiny\textcolor{black}{0}} -- ($(u)+2*(e2)$);
\draw[thick,teal,-latex] ($(u)+2*(e2)$) node [below] {\tiny\textcolor{black}{0}} -- (end2);
\draw ($(O)+(80,-20)$) node {\scriptsize (a) \; $\gamma^\perp \subseteq \rho^\perp$};

\coordinate (OC) at (150,-140);

\coordinate (O) at (430,230);
\coordinate (Ou) at (70,-130);
\coordinate (Oend1) at (10,-110);
\coordinate (Oend2) at (0,-60);
\coordinate (A) at ($(O)+(OA)$);
\coordinate (B) at ($(O)+(OB)$);
\coordinate (C) at ($(O)+(OC)$);
\coordinate (end1) at ($(O)+(Oend1)$);
\coordinate (end2) at ($(O)+(Oend2)$);

\filldraw [conefill] 
(O) node[left] {$0$} .. controls (O) and (B) .. (B) .. controls +(-15,25) and +(15,-25) .. (A) .. node [midway,left,blue] {$\rho^\perp$} controls (A) and (O) .. (O) -- cycle;
\filldraw [conefill_gamma] 
(O) .. controls (O) and (C) .. (C) .. node [midway,above,red] {$\gamma^\perp$} controls +(-25,15) and +(25,-15) .. (B) .. controls (B) and (O) .. (O) -- cycle;
\draw[very thick,violet] (O) -- (B); 

\coordinate [label=right:{\textcolor{red}{\scriptsize $u \in \gamma^\perp \setminus \rho^\perp$}}] (u) at ($(O)+(Ou)$);
\coordinate (e1) at ($1/3*(end1)-1/3*(u)$);
\coordinate (e2) at ($1/3*(end2)-1/3*(u)$);
\draw[thick,teal,-latex] (u) node [below right] {\tiny\textcolor{black}{$\ne 0$}} -- node [midway,above] {\scriptsize $e_1$} ($(u)+(e1)$);
\draw[thick,teal,-latex] ($(u)+(e1)$) node [above] {\tiny\textcolor{black}{0}} -- ($(u)+2*(e1)$);
\draw[thick,teal,-latex] ($(u)+2*(e1)$) node [above] {\tiny\textcolor{black}{0}} -- (end1) node [above] {\tiny\textcolor{black}{0}};
\draw[thick,teal,-latex] (u) -- node [midway,below] {\scriptsize $e_2$} ($(u)+(e2)$);
\draw[thick,teal,-latex] ($(u)+(e2)$) node [below] {\tiny\textcolor{black}{0}} -- ($(u)+2*(e2)$);
\draw[thick,teal,-latex] ($(u)+2*(e2)$) node [below] {\tiny\textcolor{black}{0}} -- (end2) node [below] {\tiny\textcolor{black}{0}};
\node[draw,circle,red,inner sep=1pt,fill] at (u) {};
\node[draw,circle,blue!60,inner sep=1pt,fill] at (end1) {};
\node[draw,circle,blue!60,inner sep=1pt,fill] at (end2) {};
\draw ($(O)+(80,-20)$) node {\scriptsize (b) \; $e_1, e_2 \notin \gamma^\perp$};

\coordinate (O) at (120,430);
\coordinate (Ou) at (50,-130);
\coordinate (Oend1) at ($0.75*(OB)$);
\coordinate (Oend2) at ($0.4*(OB)$);
\coordinate (A) at ($(O)+(OA)$);
\coordinate (B) at ($(O)+(OB)$);
\coordinate (C) at ($(O)+(OC)$);
\coordinate (end1) at ($(O)+(Oend1)$);
\coordinate (end2) at ($(O)+(Oend2)$);

\filldraw [conefill] 
(O) node[left] {$0$} .. controls (O) and (B) .. (B) .. controls +(-15,25) and +(15,-25) .. (A) .. node [midway,left,blue] {$\rho^\perp$} controls (A) and (O) .. (O) -- cycle;
\filldraw [conefill_gamma] 
(O) .. controls (O) and (C) .. (C) .. node [midway,above,red] {$\gamma^\perp$} controls +(-25,15) and +(25,-15) .. (B) .. controls (B) and (O) .. (O) -- cycle;
\draw[very thick,violet] (O) -- (B) node [above] {\small $\rho^\perp \cap \gamma^\perp$};

\coordinate [label=right:{\textcolor{red}{\scriptsize $\langle p, u\rangle = 1$}},label=above right:{\textcolor{red}{\scriptsize $u \in \gamma^\perp$}}] (u) at ($(O)+(Ou)$);
\coordinate (e1) at ($(end1)-(u)$);
\coordinate (e2) at ($(end2)-(u)$);
\draw[thick,red,-latex] (u) node [below right] {\tiny\textcolor{black}{$\ne 0$}} -- node [midway,above] {\scriptsize $e_1$} ($(u)+(e1)$) node [left] {\tiny\textcolor{black}{1}};
\draw[thick,red,-latex] (u) -- node [midway,right] {\scriptsize $e_2$} ($(u)+(e2)$) node [left] {\tiny\textcolor{black}{1}};
\node[draw,circle,red,inner sep=1pt,fill] at (u) {};
\node[draw,circle,violet!60,inner sep=1pt,fill] at (end1) {};
\node[draw,circle,violet!60,inner sep=1pt,fill] at (end2) {};
\draw ($(O)+(80,-20)$) node {\scriptsize (c) \; $e_1, e_2 \in \gamma^\perp$};

\coordinate (O) at (300,430);
\coordinate (Ou) at (100,-80);
\coordinate (Oend1) at (-5,-110);
\coordinate (Oend2) at (0,-60);
\coordinate (A) at ($(O)+(OA)$);
\coordinate (B) at ($(O)+(OB)$);
\coordinate (C) at ($(O)+(OC)$);
\coordinate (end1) at ($(O)+(Oend1)$);
\coordinate (end2) at ($(O)+(Oend2)$);

\filldraw [conefill] 
(O) node[left] {$0$} .. controls (O) and (B) .. (B) .. controls +(-15,25) and +(15,-25) .. (A) .. node [midway,left,blue] {$\rho^\perp$} controls (A) and (O) .. (O) -- cycle;
\filldraw [conefill_gamma] 
(O) .. controls (O) and (C) .. (C) .. node [midway,above,red] {$\gamma^\perp$} controls +(-25,15) and +(25,-15) .. (B) .. controls (B) and (O) .. (O) -- cycle;
\draw[very thick,violet] (O) -- (B); 

\coordinate [label=right:{\textcolor{teal}{\scriptsize $u \notin \gamma^\perp$}}] (u) at ($(O)+(Ou)$);
\coordinate (e1) at ($1/3*(end1)-1/3*(u)$);
\coordinate (e2) at ($1/3*(end2)-1/3*(u)$);
\draw[thick,teal,-latex] (u) node [below right] {\tiny\textcolor{black}{$0$}} -- node [midway,above] {\scriptsize $e_2$} ($(u)+(e1)$);
\draw[thick,teal,-latex] ($(u)+(e1)$) node [above] {\tiny\textcolor{black}{0}} -- ($(u)+2*(e1)$);
\draw[thick,teal,-latex] ($(u)+2*(e1)$) node [above] {\tiny\textcolor{black}{0}} -- (end1) node [above] {\tiny\textcolor{black}{0}};
\draw[thick,teal,-latex] (u) -- node [midway,below] {\scriptsize $e_1$} ($(u)+(e2)$);
\draw[thick,teal,-latex] ($(u)+(e2)$) node [below] {\tiny\textcolor{black}{0}} -- ($(u)+2*(e2)$);
\draw[thick,teal,-latex] ($(u)+2*(e2)$) node [below] {\tiny\textcolor{black}{0}} -- (end2) node [below] {\tiny\textcolor{black}{0}};
\node[draw,circle,teal,inner sep=1pt,fill] at (u) {};
\node[draw,circle,blue!60,inner sep=1pt,fill] at (end1) {};
\node[draw,circle,blue!60,inner sep=1pt,fill] at (end2) {};
\draw ($(O)+(90,-20)$) node {\scriptsize (d) \; $e_1 \notin \gamma^\perp, e_2 \in \gamma^\perp$};

\coordinate (O) at (500,430);
\coordinate (Ou) at (70,-110);
\coordinate (Oend1) at ($0.75*(OB)$);
\coordinate (Oend2) at (0,-60);
\coordinate (A) at ($(O)+(OA)$);
\coordinate (B) at ($(O)+(OB)$);
\coordinate (C) at ($(O)+(OC)$);
\coordinate (end1) at ($(O)+(Oend1)$);
\coordinate (end2) at ($(O)+(Oend2)$);

\filldraw [conefill] 
(O) node[left] {$0$} .. controls (O) and (B) .. (B) .. controls +(-15,25) and +(15,-25) .. (A) .. node [midway,left,blue] {$\rho^\perp$} controls (A) and (O) .. (O) -- cycle;
\filldraw [conefill_gamma] 
(O) .. controls (O) and (C) .. (C) .. node [midway,above,red] {$\gamma^\perp$} controls +(-25,15) and +(25,-15) .. (B) .. controls (B) and (O) .. (O) -- cycle;
\draw[very thick,violet] (O) -- (B) node [above] {\small $\rho^\perp \cap \gamma^\perp$};

\coordinate [label=right:{\textcolor{red}{\scriptsize $u \in \gamma^\perp$}}] (u) at ($(O)+(Ou)$);
\coordinate (e1) at ($1/3*(end1)-1/3*(u)$);
\coordinate (e2) at ($1/3*(end2)-1/3*(u)$);
\draw[thick,red,-latex] (u) node [below right] {\tiny\textcolor{black}{$\ne 0$}} -- node [midway,above] {\scriptsize $e_2$} ($(u)+(e1)$);
\draw[thick,red,-latex] ($(u)+(e1)$) node {} -- ($(u)+2*(e1)$);
\draw[thick,red,-latex] ($(u)+2*(e1)$) node {} -- (end1) node [left] {\tiny\textcolor{black}{1}};
\draw[thick,teal,-latex] (u) -- node [midway,below] {\scriptsize $e_1$} ($(u)+(e2)$);
\draw[thick,teal,-latex] ($(u)+(e2)$) node [below] {\tiny\textcolor{black}{0}} -- ($(u)+2*(e2)$);
\draw[thick,teal,-latex] ($(u)+2*(e2)$) node [below] {\tiny\textcolor{black}{0}} -- (end2) node [below] {\tiny\textcolor{black}{0}};
\node[draw,circle,red,inner sep=1pt,fill] at (u) {};
\node[draw,circle,violet!60,inner sep=1pt,fill] at (end1) {};
\node[draw,circle,blue!60,inner sep=1pt,fill] at (end2) {};
\draw ($(O)+(90,-20)$) node {\scriptsize (d) \; $e_1 \notin \gamma^\perp, e_2 \in \gamma^\perp$};

\end{tikzpicture}
\end{center}
\caption{The dual cone in four cases of Theorem~\ref{idemp_theor}.}
\label{4cases_pict}
\end{figure}

Recall that from $x \in O_\gamma$ it follows that $\chi^u(x) \ne 0$ if and only if $u \in \gamma^\perp$. 

(a) Let $\rho$ be a ray of the face $\gamma$. Then $\gamma^\perp \subseteq \rho^\perp$, so, according to the above, we obtain $\chi^u(x) = \begin{cases}1 \text{ if } u \in \gamma^\perp\\ 0 \text{ otherwise}\end{cases}$, i.e. $x = x_\gamma$. 

Conversely, let us check that $x = x_\gamma$ is an idempotent, i.e. $\chi^u(x_\gamma*x_\gamma) = \chi^u(x_\gamma)$ for any $u \in S_\sigma$. If $u \in \rho^\perp$, then the sum in equation~\eqref{comult_X_binom_eq2} consists of one summand and $\chi^u(x_\gamma*x_\gamma) = \chi^u(x_\gamma)\chi^u(x_\gamma) = \chi^u(x_\gamma)$ since $\chi^u(x_\gamma) = 0,1$. 
If $u \notin \rho^\perp$, then the sum in equation~\eqref{comult_X_binom_eq2} consists of more than one summand, but in each summand $u+ie_2$ or $u+ie_1$ do not belong to $\rho^\perp$, i.e. each summand equals zero. So if $u \notin \rho^\perp$, then $\chi^u(x_\gamma*x_\gamma) = 0 = \chi^u(x_\gamma)$. 

\smallskip

(b) Let $\rho$ be not a ray of the face $\gamma$ and $e_1, e_2 \notin \gamma^\perp$. Since $e_1 \notin \gamma$, there exists a ray $\rho'$ of~$\gamma$ with the primitive vector~$p'$ such that $\langle p', e_1\rangle \ne 0$. Moreover, $\rho' \ne \rho$ since $\rho$ is not a ray of $\gamma$, so $\langle p', e_1\rangle > 0$. 

Take any $u \in \gamma^\perp \setminus \rho^\perp$. Then for any $1 \le j \le \langle p,u\rangle$ we have $\langle p', u+je_1\rangle = j\langle p', e_1\rangle > 0$, so $u + je_1 \notin \gamma^\perp$ and $\chi^{u+je_1}(x) = 0$. In the same way, $\chi^{u+ie_2}(x) = 0$ for $1 \le i \le \langle p,u\rangle$. Since $u \notin \rho^\perp$, it follows that the sum in $\chi^u(x*x)$ has at least two summands and equals zero. However, $\chi^u(x) \ne 0$ since $u \in \gamma^\perp$. Thus there is no idempotents in case (b). 

\smallskip

(c) Let $\rho$ be not a ray of the face $\gamma$ and $e_1, e_2 \in \gamma^\perp$. Take $u \in \gamma^\perp$ such that $\langle p,u\rangle = 1$. Such $u \in S_\sigma$ exists since the rational plane $\gamma^\perp \cap \{\langle p, \,\cdot\,\rangle = 1\}$ contains a lattice element~$-e_1$, so one can find a lattice element in the intersection of this plane and the cone~$\sigma^\vee$ as well. 

Then $u+e_1, u+e_2 \in \gamma^\perp \cap \rho^\perp$, whence $\chi^{u+e_1}(x) = \chi^{u+ie_2}(x) = 1$ and formula~\eqref{comult_X_binom_eq2} gives for an idempotent $x$ the equation $\chi^u(x) = \chi^u(x*x) = \chi^u(x) + \chi^u(x)$. It follows that $\chi^u(x)=0$, which contradicts $u \in \gamma^\perp$. 

\smallskip

(d) Let $\rho$ be not a ray of the face $\gamma$ and $e_1 \notin \gamma^\perp$, $e_2 \in \gamma^\perp$. From $e_1 \notin \gamma^\perp$ it follows that there exists a ray of $\gamma$ with the primitive vector $p'$ such that $\langle p', e_1\rangle > 0$, see case~(b). Consider two cases. 

Let $u \notin \gamma^\perp$. Since $p' \in \gamma$ and $u \in \sigma^\vee \setminus \gamma^\perp$, we have
$$\begin{aligned}
\langle p', u+je_1\rangle &> 0+j\langle p',e_1\rangle \ge 0 \quad\text{ for any } 0 \le j \le \langle p,u\rangle,\\
\langle p', u+ie_2\rangle &> 0+i\langle p',e_2\rangle \ge 0 \quad\text{ for any } 0 \le i \le \langle p,u\rangle,
\end{aligned}$$
so 
$$\begin{aligned}
u+je_1 &\notin \gamma^\perp \quad\text{ for any } 0 \le j \le \langle p,u\rangle,\\
u+ie_2 &\notin \gamma^\perp \quad\text{ for any } 0 \le i \le \langle p,u\rangle. 
\end{aligned}$$ 
Then $\chi^u(x) = 0 = \chi^u(x*x)$ for any $u \notin \gamma^\perp$. 

Let $u \in \gamma^\perp$. In this case, we have 
$$\begin{aligned}
\langle p', u+je_1\rangle &= 0+j\langle p',e_1\rangle > 0 \quad\text{ for any } 1 \le j \le \langle p,u\rangle,\\
\langle p', u+ie_2\rangle &= 0+i\langle p',e_2\rangle = 0 \quad\text{ for any } 0 \le i \le \langle p,u\rangle,
\end{aligned}$$
i.e.
$$\begin{aligned}
u+je_1 &\notin \gamma^\perp \quad\text{ for any } 1 \le j \le \langle p,u\rangle,\\
u+ie_2 &\in \gamma^\perp \quad\text{ for any } 0 \le i \le \langle p,u\rangle.
\end{aligned}$$
Then for $x \in O_\gamma$, the condition $\chi^u(x) = \chi^u(x*x)$ is equivalent to $\chi^u(x) = \chi^u(x)\chi^{u+\langle p,u\rangle e_2}(x)$, i.e. $\chi^{u+\langle p,u\rangle e_2}(x) = 1$. Notice that $u+\langle p,u\rangle e_2 \in \rho^\perp \cap \gamma^\perp = \cone(\gamma, \rho)^\perp$, so the condition $\chi^{u+\langle p,u\rangle e_2}(x) = 1$ is equivalent~\eqref{idemp_rhoperp_eq}. This completes the proof. 
\end{proof}

\section{Geometry of the set $E(X)$}

Let us take a closer look at the geometric structure of the set of idempotents. Denote the set of all idempotents in~$X$ by~$E(X)$. It is closed in~$X$. The set of idempotents $E_\gamma$ in $O_\gamma$ in item~(d) of Theorem~\ref{idemp_theor} is closed in $O_\gamma$ and may be not closed in $X$. We are going to prove that the closure of $E_\gamma$ in $X$ is the union of $E_\gamma$ and one point, which is an idempotent from item~(a) of Theorem~\ref{idemp_theor}. Idempotents are also connected with the action of the group $G_\chi \times G_\chi$ by left and right multiplication, where $G_\chi$ is the group of invertible elements in the monoid. More precisely, the following statement holds. 

\begin{proposition}
\label{idemp_geom_prop}
In the notation of Theorems~\ref{classif_theor} and~\ref{idemp_theor}, the following assertions hold. 
\begin{enumerate}
  \item For any face $\gamma$ of the cone $\sigma$ such that $\rho$ is not a ray of $\gamma$ and exactly one of Demazure roots $e_1, e_2$ belongs to $\gamma^\perp$
  \[\overline{E_\gamma} = E_\gamma \cup \{x_{\cone(\gamma, \rho)}\}.\]
  \item Irreducible components of the subvariety $E(X)$ do not intersect, each of them is either a point or is isomorphic to the affine line. 
  \item Any irreducible component of $E(X)$ is a subset of a $(G_\chi \times G_\chi)$-orbit, and any ${(G_\chi \times G_\chi)}$-orbit contains at most one irreducible component of $E(X)$. 
\end{enumerate}
\end{proposition}

Before a proof we recall some results of~\cite{AKZ2012}. Let $X = X_\sigma$ be an affine toric variety with an acting torus~$\TT$, $p$ be the primitive vector on a ray of~$\sigma$, and $e$ be a Demazure root corresponding to~$p$. Denote by $H_e$ the one-parameter subgroup in $\Aut(X)$ normalized by the torus~$\TT$ and corresponding to~$e$, i.e. to a homogeneous locally nilpotent derivation with degree~$e$, see subsection~\ref{LND_Dem_subsec}. 

In \cite[Proposition~2.1]{AKZ2012} it is proved that any orbit of the group $H_e$ in $X$ is either a point or intersects exactly two $\TT$-orbits $O_1$, $O_2$, where $\dim O_1 = \dim O_2 + 1$ and $O_2 \subseteq \overline{O_1}$. Moreover, in the latter case
\begin{equation}
\label{Heorbit_eq}
\begin{aligned}
O_1 \cap H_ex = R_px,\\
O_2 \cap H_ex = \{pt\},
\end{aligned}
\end{equation}
where $x$ is any point in $O_1$ and $R_p \cong \KK^\times$ is the one-parameter subgroup in $\Aut(X)$ given by the primitive vector $p \in N$. Such a pair of $\TT$-orbits $(O_1, O_2)$ is called \emph{$H_e$-connected}. 

Recall that there is a one-to-one correspondence between faces $\tau$ of the cone $\sigma$ and $\TT$-orbits $O_\tau$ on $X$, see subsection~\ref{Torb_subsec}. In~\cite[Lemma 2.2]{AKZ2012}, a combinatorial condition on faces of $\sigma$ corresponding to $H_e$-connected orbits is given. Namely, a pair of $\TT$-orbits $(O_{\tau_1}, O_{\tau_2})$ is $H_e$-connected if and only if $\langle \tau_2, e\rangle \le 0$, the cone $\tau_1$ is a facet of $\tau_2$, and $\tau_1 = \tau_2 \cap e^\perp$. 

\begin{proof}[Proof of Proposition~\ref{idemp_geom_prop}]
Let $\gamma$ be a face of~$\sigma$ and $e$ be a Demazure root corresponding to the ray $\rho$ of $\sigma$. First we are going to note that 
\begin{equation}
\label{idempHe_lemeq}
\text{for a face $\gamma'$ the pair $(O_\gamma, O_{\gamma'})$ is $H_e$-connected} \quad \Leftrightarrow \quad 
\left\{\begin{aligned}
&\text{$\rho$ is not a ray of~$\gamma$},\\
&\gamma' = \cone(\gamma, \rho),\\
&e \in \gamma^\perp.
\end{aligned}\right.
\end{equation}

Indeed, if the pair $(O_\gamma, O_{\gamma'})$ is $H_e$-connected, then $\gamma$ is a facet of $\gamma'$ by~\cite[Lemma 2.2]{AKZ2012}, so $\gamma' = \cone(\gamma, \rho_1, \ldots, \rho_k)$ for some rays $\rho_i$ of~$\sigma$, where $\rho_i \nsubseteq \gamma$. By~\cite[Lemma 2.2]{AKZ2012}, $\gamma = \gamma' \cap e^\perp$ and $\langle \gamma', e\rangle \le 0$, so $\langle \rho_i,e\rangle < 0$. Then according to the definition of a Demazure root $\gamma' = \cone(\gamma, \rho)$. Also, $\rho$ is not a ray of $\gamma$ since $\gamma$ is a facet of $\gamma'$, and $\gamma = \gamma' \cap e^\perp$ implies $e \in \gamma^\perp$. 

Conversely, let the right-hand side of~\eqref{idempHe_lemeq} hold. Then it follows from the definition of a Demazure root that $\gamma'$ is a face of~$\sigma$, see Remark~\ref{def_Dem_cone_rem}. Since $\langle\gamma,e\rangle = 0$ and $\langle\rho,e\rangle \le 0$, it is also easy to see that conditions of~\cite[Lemma 2.2]{AKZ2012} hold, so the pair $(O_\gamma, O_{\gamma'})$ is $H_e$-connected. 

Let $\gamma$ be a face from case~(d) of Theorem~\ref{idemp_theor}. Without loss of generality, one can assume that $e_1 \in \gamma^\perp$ and $e_2 \notin \gamma^\perp$. Consider the set of idempotents \[E_\gamma = O_\gamma \cap \,\{\chi^u = 1 \;\; \forall u \in \cone(\gamma, \rho)^\perp \cap S_\sigma\}.\] It is invariant with respect to~$R_p$ since the action of~$R_p$ on~$\KK[X]$ is defined by the formula $t \cdot \chi^u = t^{\langle p,u\rangle}\chi^u$ and $\langle p,u\rangle = 0$ for any $u \in \cone(\gamma, \rho)^\perp$. So $E_\gamma$ is a union of some $R_p$-orbits. 

According to~\eqref{idempHe_lemeq}, the pair $(O_\gamma, O_{\gamma'})$ is $H_{e_1}$-connected. Since $H_{e_1}$ is unipotent and $X$ is affine, any $H_{e_1}$-orbit is closed in $X$, see~\cite[Section~1.3]{PV1994}. So taking the closure of any $R_p$-orbit in $O_\gamma$ adds one point, which belongs to $O_{\gamma'}$, and gives an $H_{e_1}$-orbit by~\cite[Proposition~2.1]{AKZ2012}, see equation~\eqref{Heorbit_eq}. We apply this to $R_p$-orbits in $E_\gamma$. The set of idempotents in $X$ is closed, so the added points in $O_{\gamma'}$ are idempotents in $X$ as well. However, $\rho$ is a ray of the face $\gamma'=\cone(\gamma, \rho)$, so according to case (a) of Theorem~\ref{idemp_theor} there is exactly one idempotent $x_{\gamma'}$ in $O_{\gamma'}$. It follows that $E_\gamma$ consists of one $R_p$-orbit and gives the first statement of Proposition~\ref{idemp_geom_prop}. 

The group of invertible elements in $X$ equals $G_\chi = \GG_a \leftthreetimes \TT$, where left and right multiplications by~$\GG_a$ on~$X$ correspond to $e_1$ and $e_2$. So the orbits of the natural action of the group $G_\chi \times G_\chi$ on~$X$ are the unions of $H_{e_1}$- and $H_{e_2}$-connected $\TT$-orbits. So if a face $\gamma$ corresponds to case (b) of Theorem~\ref{idemp_theor}, then $O_\gamma$ is a $(G_\chi\times G_\chi)$-orbit and there is no idempotents in $O_\gamma$. If a face $\gamma$ corresponds to case (c) of Theorem~\ref{idemp_theor}, then $O_\gamma \cup O_{\cone(\gamma, \rho)}$ is a $(G_\chi\times G_\chi)$-orbit and there is one idempotent $x_{\cone(\gamma, \rho)}$ in $O_\gamma$. For a face $\gamma$ in case (d), $O_\gamma \cup O_{\cone(\gamma, \rho)}$ is a $(G_\chi\times G_\chi)$-orbit and the set of idempotents in it equals $\overline{E_{\gamma}}$, which is isomorphic to $H_e \cong \AA^1$. If $\gamma'$ is a face from case~(a) such that there is no $\gamma$ from cases (b) or (d) with $\gamma'=\cone(\gamma, \rho)$, then $O_\gamma$ is a $(G_\chi\times G_\chi)$-orbit with one idempotent $x_{\gamma'}$ in it. This completes the proof of Proposition~\ref{idemp_geom_prop}. 
\end{proof}

Recall that an element $\zero$ in a monoid $X$ is called the \emph{zero} if $\zero*x = x*\zero = \zero$ for any $x \in X$. Clearly, the zero element is unique if it exists. Note that $\zero \in E(X)$ and $\{\zero\}$ is one of $(G_\chi\times G_\chi)$-orbits, so $\zero$ is an isolated point in $E(X)$ by Proposition~\ref{idemp_geom_prop}(c) if it exists. 

\begin{proposition}
\label{zero_prop}
In the notation of Theorem~\ref{classif_theor}, the monoid $X$ has zero if and only if $\sigma^\perp = 0$ and $-e_1, -e_2 \notin \sigma^\vee$. In such a case, the zero element equals~$x_\sigma$. 
\end{proposition}

\begin{proof}
An element $\zero$ is a zero element if and only if $\chi^u(x*\zero) = \chi^u(\zero*x) = \chi^u(\zero)$ for any $u \in S_\sigma$ and $x \in X$. It holds for $u=0$ since $\chi^0 = 1$; let $u \ne 0$. Consider the sum
\[\chi^u(x*\zero)-\chi^u(\zero) = \left(\chi^{u+\langle p,u\rangle e_2}(x)-1\right)\chi^u(\zero) + \sum\limits_{\substack{i+j=\langle p, u\rangle\\j \ne 0}} \begin{pmatrix}\langle p, u\rangle\\i\end{pmatrix} \chi^{u+ie_2}(x) \chi^{u+je_1}(\zero).
\]
It equals zero for any $x\in X$ if and only if 
\begin{equation}
\label{zero_eq1}
\chi^{u+je_1}(\zero) = 0 \quad\text{ for any }\; 0 \le j \le \langle p,u\rangle.
\end{equation}
Indeed, for $u \ne 0$ the functions $\chi^{u+\langle p,u\rangle e_2}-1, \, \chi^{u+ie_2}$, $0 \le i < \langle p,u\rangle$, are linearly independent in $\KK[X] = \bigoplus_{v \in S_\sigma} \KK\chi^v$. In particular, it follows that
\[\chi^u(\zero) = \begin{cases}0 \; \text{ if } u \ne 0,\\ 1 \; \text{ if } u = 0,\end{cases}\]
i.e. $\zero = x_\sigma$ and $\sigma^\perp = 0$, see subsection~\ref{Torb_subsec}. Conversely, the point $\zero = x_\sigma$ satisfies equations~\eqref{zero_eq1} if and only if $u + ie_1 \ne 0$ for any $0 \le i \le \langle p,u\rangle$ and any $u \in S_\sigma \setminus \{0\}$, which is equivalent to $-e_1 \notin \sigma^\vee$. Similarly, $\chi^u(\zero*x)-\chi^u(\zero) = 0$ if and only if $\zero = x_\sigma$ and $-e_2 \notin \sigma^\vee$. 
\end{proof}

\begin{corollary}
\label{zeroAn_cor}
In the notation of Theorem~\ref{affspace_theor}, the monoid on $\AA^n$ admits no zero element if and only if $a_1 = \ldots = a_{n-1} = 0$ or $b_1 = \ldots = b_{n-1} = 0$. Otherwise, the zero element is $(0, \ldots, 0) \in \AA^n$.
\end{corollary}

For example, there is no zero element in the monoid from Example~\ref{idempA2_example}, while the point $(0,0,0,0)\in\AA^4$ is the zero element in Example~\ref{idempA4_example}. 

\begin{proof}
According to the proof of Theorem~\ref{affspace_theor}, two Demazure roots are $e_1 = (b_1, \ldots, b_{n-1}, -1)$, $e_2 = (a_1, \ldots, a_{n-1}, -1)$, $a_1, \ldots, a_{n-1}, b_1, \ldots, b_{n-1} \in \ZZ_{\ge0}$, and $\sigma$ and $\sigma^\vee$ are the positive octants in $N_\QQ$ and $M_\QQ$, respectively. Then Proposition~\ref{zero_prop} implies the required condition. 
\end{proof}

\section{The center}

Consider the \emph{center} of a monoid $X$:
\[Z(X) = \{x \in X \mid x*y = y*x \;\;\forall y \in X\}.\]
We find equations that define $Z(X)$ and study the dimensions of its irreducible components. 

\begin{proposition}
\label{centerX_prop}
In the notation of Theorem~\ref{classif_theor}, we have 
\[Z(X) = \begin{cases}
  \overline{O_\rho} \cap \{\chi^{u+e_1} = \chi^{u+e_2} \;\;\;\forall u \in S_\sigma: \langle p,u\rangle = 1\} \quad \text{if}\; e_1 \ne e_2;\\
  X \quad \text{if}\; e_1 = e_2.
  \end{cases}\]
\end{proposition}

\begin{remark}
\label{centereq_rem}
One can leave only equations $\chi^{u_i+e_1} = \chi^{u_i+e_2}$, where 
\[\{u \in S_\sigma: \langle p,u\rangle = 1\} \subseteq \bigcup_i (u_i + S_\sigma).\]
Indeed, for any $v \in S_\sigma$ the equation $\chi^{u+e_1}=\chi^{u+e_2}$ implies $\chi^{v+u+e_1}=\chi^{v+u+e_2}$. 
\end{remark}

\begin{proof}
If $e_1 = e_2$, then $X$ is commutative; from now on $e_1 \ne e_2$. Since the set of invertible elements $G_\chi$ is dense in~$X$, 
\[Z(X) = \{x \in X \mid x*y = y*x \;\;\forall y \in G_\chi\} = \{x \in X \mid x = y*x*y^{-1} \;\;\forall y \in G_\chi\}.\]

First let us check that for $y \in G_\chi$ the inverse element is given by
\begin{equation}
\label{chiu_inverse_eq}
\chi^u(y^{-1}) = (-1)^{\langle p,u\rangle}\chi^{-u-\langle p,u\rangle(e_1+e_2)}(y), \quad u \in S_\sigma.
\end{equation}
It us easy to see that the map that takes $u$ to the right-hand side of~\eqref{chiu_inverse_eq} is a homomorphism of semigroups $S_\sigma \to \KK^\times$, so formula~\eqref{chiu_inverse_eq} defines a point $z$ in $X$. Let us multiply $y$ by $z$. According to equation~\eqref{comult_X_binom_eq2},
\[\chi^u(y*z) = 
\sum\limits_{i+j=\langle p, u\rangle} \begin{pmatrix}\langle p, u\rangle\\i\end{pmatrix} \chi^{u+ie_2}(y)\chi^{u+je_1}(z).\]
For $u' = u+je_1$ we have $\langle p, u'\rangle = \langle p,u\rangle - j = i$ and $-u'-\langle p,u'\rangle(e_1+e_2) = -u-ie_2-\langle p,u\rangle e_1$, so
\begin{multline*}
\chi^u(y*z) = 
\sum\limits_{i+j=\langle p, u\rangle} \begin{pmatrix}\langle p, u\rangle\\i\end{pmatrix} (-1)^i\chi^{u+ie_2}(y)\chi^{-u-ie_2-\langle p,u\rangle e_1}(y) = 
\\
= \chi^{-\langle p,u\rangle e_1}(y) \sum\limits_{i+j=\langle p, u\rangle} \begin{pmatrix}\langle p, u\rangle\\i\end{pmatrix} (-1)^i = 
\begin{cases}
1, \;\; \text{if} \;\; \langle p,u\rangle = 0\\
0, \;\; \text{if} \;\; \langle p,u\rangle \ne 0
\end{cases}
= \chi^u(x_\rho),
\end{multline*}
where $x_\rho$ is a unity in $X$, see subsections~\ref{Torb_subsec} and~\ref{invproof_subsect}. Thus $y * z = x_\rho$, in the same way $z * y = x_\rho$, whence $z = y^{-1}$. 

Now we can calculate $\chi^u(y*x*y^{-1})$ for $u \in S_\sigma$: 
\begin{multline*}
\chi^u(y*x*y^{-1}) = 
\sum\limits_{i+j=\langle p, u\rangle} \begin{pmatrix}\langle p, u\rangle\\i\end{pmatrix}\chi^{u+ie_2}(y*x)\chi^{u+je_1}(y^{-1}) = 
\\
= \sum\limits_{i+j=\langle p, u\rangle} \begin{pmatrix}\langle p, u\rangle\\i\end{pmatrix} (-1)^i\chi^{u+ie_2}(y*x)\chi^{-u-ie_2-\langle p,u\rangle e_1}(y) = 
\\
= \sum\limits_{i+k+l=\langle p, u\rangle} \begin{pmatrix}\langle p, u\rangle\\i\end{pmatrix}\begin{pmatrix}\langle p, u\rangle - i\\k\end{pmatrix}
(-1)^i\chi^{u+ie_2+ke_2}(y)\chi^{u+ie_2+le_1}(x)\chi^{-u-ie_2-\langle p,u\rangle e_1}(y) = 
\\
= \sum\limits_{i+k+l=\langle p, u\rangle} 
\frac{\langle p,u\rangle!}{i!k!l!} (-1)^i\chi^{u+ie_2+le_1}(x)\chi^{ke_2-\langle p,u\rangle e_1}(y) = 
\\
= \sum\limits_{k=0}^{\langle p, u\rangle} \begin{pmatrix}\langle p, u\rangle\\k\end{pmatrix}
\left(\sum\limits_{i+l=\langle p, u\rangle - k} \begin{pmatrix}\langle p, u\rangle - k\\i\end{pmatrix}(-1)^i\chi^{u+ie_2+le_1}(x)\right)\chi^{ke_2-\langle p,u\rangle e_1}(y) = 
\\
= \sum\limits_{k=0}^{\langle p, u\rangle} \begin{pmatrix}\langle p, u\rangle\\k\end{pmatrix}
\left(\chi^u(\chi^{e_1} - \chi^{e_2})^{\langle p, u\rangle - k}\right)(x) \cdot \chi^{ke_2-\langle p,u\rangle e_1}(y).
\end{multline*}
Note that the summand for $k = \langle p,u\rangle$ equals $\chi^u(x)\chi^{\langle p,u\rangle(e_2-e_1)}(y)$. A point $x$ belongs to $Z(X)$ if and only if $\chi^u(y*x*y^{-1}) - \chi^u(x) = 0$ for any $y \in G_\chi$, i.e.
\begin{multline*}
\chi^u(x)\left(\chi^{\langle p,u\rangle(e_2-e_1)}(y) - 1\right) +
\sum\limits_{k=0}^{\langle p, u\rangle - 1} \begin{pmatrix}\langle p, u\rangle\\k\end{pmatrix}
\left(\chi^u(\chi^{e_1} - \chi^{e_2})^{\langle p, u\rangle - k}\right)(x) \cdot \chi^{ke_2-\langle p,u\rangle e_1}(y) = 0
\end{multline*}
for any $y \in G_\chi$. By Theorem~\ref{classif_theor}, the set of invertible elements in~$X$ equals the open subset $X_\rho = G_\chi$. 
If $\langle p,u\rangle \ne 0$, then the functions $\chi^{\langle p,u\rangle(e_2-e_1)}(y) - 1, \chi^{ke_2-\langle p,u\rangle e_1}(y)$, where $0 \le k \le \langle p,u\rangle - 1$, are linearly independent in $\KK[X_\rho]= \bigoplus_{v \in S_\rho} \KK\chi^v$. If $\langle p,u\rangle \ne 0$, then the first function is zero and other are linearly independent. 
So the equations on a point~$x$ in $Z(X)$ are as follows:
\begin{equation}
\label{center_eq}
\left\{\begin{aligned}
\chi^u(x) &= 0 \quad \text{if} \; \langle p,u\rangle \ne 0; \\
\chi^u(\chi^{e_1} - \chi^{e_2})^m &= 0 \quad \text{for} \; 1 \le m \le \langle p,u\rangle.
\end{aligned}
\right.
\end{equation}

The first line of~\eqref{center_eq} is equivalent to $x \in \overline{O_\rho}$, see subsection~\ref{Torb_subsec}. Let us show that the equation for $m=1$ implies all other in the second line of~\eqref{center_eq}. We use induction on~$m$; let $m \ge 2$. For shortness we assume $\binom{m}{i} = 0$ if $i < 0$ or $i > m$. Then
\begin{multline}
\chi^u(\chi^{e_1} - \chi^{e_2})^m = 
\sum\limits_{i+l=m} \begin{pmatrix}m\\i\end{pmatrix}(-1)^i\chi^{u+ie_2+le_1} = 
\\
= \sum\limits_{i=0}^m \left(\begin{pmatrix}m-1\\i\end{pmatrix} + \begin{pmatrix}m-1\\i-1\end{pmatrix}\right)(-1)^i\chi^{u+ie_2+(m-i)e_1} = 
\\
= \sum\limits_{j=0}^{m-1} \begin{pmatrix}m-1\\j\end{pmatrix}\left((-1)^j\chi^{u+je_2+(m-j)e_1} + (-1)^{j+1}\chi^{u+(j+1)e_2+(m-j-1)e_1}\right),
\end{multline}
where any summand equals $0$ by induction hypothesis for $u+je_2+(m-j-1)e_1$. 

Moreover, for any $u \in S_\sigma$ with $\langle p, u\rangle > 1$ and $x \in \overline{O_\rho}$, we have $\chi^{u+e_1}(x) = 0 = \chi^{u+e_2}(x)$ since $\chi^v(x) = 0$ if $\langle p, v\rangle \ne 0$. Thus, we need only equations $\chi^{u+e_1}-\chi^{u+e_2}=0$ for ${\langle p,u\rangle=1}$. 
\end{proof}

\begin{corollary}
\label{dimZ_cor}
In the notation of Theorem~\ref{classif_theor}, we have $\dim Z(X) = \dim X - 2$ if $X$ is noncommutative. 
\end{corollary}

\begin{proof}
By Theorem~\ref{classif_theor}, the set of invertible elements in~$X$ equals the open subset $X_\rho = G_\chi$. The intersection $Z(X) \cap X_\rho$ is the center $Z(G_\chi)$ of the group of invertible elements, which is isomorphic to $\Ker \chi$ for the character~$\chi$ of the torus of dimension $\dim X-1$ by Lemma~\ref{centerG_lem}. So there exists the irreducible component $\overline{Z(G_\chi)}$ of $Z(X)$ of dimension $\dim X - 2$. 

By Proposition~\ref{centerX_prop}, we have $Z(X) \subseteq \overline{O_\rho}$. Since $X_\rho = O_\rho \cup O_0$, any other irreducible components of $Z(X)$ is a subset of $\overline{O_\rho} \setminus O_\rho$ of dimension $\dim O_\rho - 1 = \dim X - 2$, see subsection~\ref{Torb_subsec}. This completes the proof. 
\end{proof}

\begin{corollary}
\label{centerAn_cor}
In the notation of Theorem~\ref{affspace_theor}, the center of the monoid on~$\AA^n$ equals
\[Z(\AA^n) = \{x_n = 0, \;\; x_1^{a_1} \ldots x_{n-1}^{a_{n-1}} = x_1^{b_1} \ldots x_{n-1}^{b_{n-1}}\}\]
if the monoid is noncommutative. In particular, all irreducible components of $Z(\AA^n)$ are of dimension~$n-2$. 
\end{corollary}

\begin{proof}
We use the notation of the proof of Theorem~\ref{affspace_theor}. Notice that for the basis vector $u_n = (0,\ldots,0,1)\in M$ and the primitive vector $p=(0,\ldots,0,1) \in N$ on the ray~$\rho$, we have the inclusion $\{u \in S_\sigma: \langle p,u\rangle = 1\} \subseteq u_n + S_\sigma$. So, according to Proposition~\ref{centerX_prop} and Remark~\ref{centereq_rem}, the center $Z(\AA^n)$ in $\overline{O_\rho} = \{x_n = 0\}$ is given by the equation $\chi^{u_n+e_1} = \chi^{u_n+e_2}$, where $e_1 = (b_1, \ldots, b_{n-1}, -1)$, $e_2 = (a_1, \ldots, a_{n-1}, -1)$. This gives the required formula. 
\end{proof}

\begin{example}
\label{dim3center_example}
Let us provide an example of a monoid such that irreducible components of the center have different dimensions. 

Consider the monoid structure on $X=\{vw=zt\}\subseteq \AA^4$ from Example~\ref{mult_dim3_example} for Demazure roots $e_1 = (-1,0,1)$ and $e_2 = (-1,1,2)$. Let us find the center using Proposition~\ref{centerX_prop}. By definition, 
\[S_\sigma = \sigma^\vee \cap M = \{(a,b,c) \in M \mid a, \, b, \, a+c, \, b+c \ge 0\},\]
see Figure~\ref{dim3pict}. Substituting $a = 0$ and $a = 1$, we obtain two sections $\{u \in S_\sigma: \langle p,u\rangle = 0\}$ and $\{u \in S_\sigma: \langle p,u\rangle = 1\}$, see Figure~\ref{dim3center_pict}. One can see that the latter one is a subset of $(u_1+S_\sigma) \cup (u_2+S_\sigma)$ for $u_1 = (1,0,0)$ and $u_2 = (1,1,-1)$. Then, according to Remark~\ref{centereq_rem}, the center is defined by 
\[Z(X) = \overline{O_\rho} \cap \{\chi^{u_i+e_1} = \chi^{u_i+e_2}, \; i = 1,2\}.\]

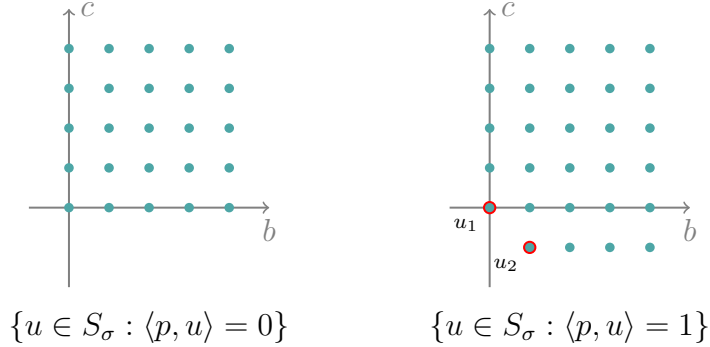
\begin{figure}[h]
\begin{center}
\tikzset{every picture/.style={line width=0.75pt}} 

\begin{tikzpicture}[x=0.75pt,y=0.75pt,yscale=-1,xscale=1]
    \tikzstyle{conefill} = [fill=blue!20, draw = black!40]
    \tikzstyle{conefill_gamma} = [fill=red!20, draw = black!70]

\coordinate (e1) at (20,0);
\coordinate (e2) at (0,-20);


\coordinate (O) at (20,230);
\coordinate (Oxmax) at ($(O)+5*(e1)$);
\coordinate (Oymax) at ($(O)+5*(e2)$);
\coordinate (Oxmin) at ($(O)-(e1)$);
\coordinate (Oymin) at ($(O)-2*(e2)$);
\coordinate (T) at ($(O)-3*(e2)+2*(e1)$);

\draw [->,color=black!50] (Oxmin) -- (Oxmax) node[below] {$b$};
\draw [->,color=black!50] (Oymin) -- (Oymax) node[right] {$c$};

\foreach \x in {0,...,4}
    \foreach \y in {0,...,4}
    {
    \node[draw=teal!70,circle,inner sep=1pt,fill=teal!70] at ($(O)+\x*(e1)+\y*(e2)$) {};
    }

\draw[black] (T) node {$\{u\in S_\sigma: \langle p,u\rangle=0\}$};


\coordinate (O) at (230,230);
\coordinate (Oxmax) at ($(O)+5*(e1)$);
\coordinate (Oymax) at ($(O)+5*(e2)$);
\coordinate (Oxmin) at ($(O)-(e1)$);
\coordinate (Oymin) at ($(O)-2*(e2)$);
\coordinate (T) at ($(O)-3*(e2)+2*(e1)$);

\draw [->,color=black!50] (Oxmin) -- (Oxmax) node[below] {$b$};
\draw [->,color=black!50] (Oymin) -- (Oymax) node[right] {$c$};

\foreach \x in {0,...,4}
    \foreach \y in {0,...,4}
    {
    \node[draw=teal!70,circle,inner sep=1pt,fill=teal!70] at ($(O)+\x*(e1)+\y*(e2)$) {};
    }
\foreach \x in {1,...,4}
{
    \node[draw=teal!70,circle,inner sep=1pt,fill=teal!70] at ($(O)+\x*(e1)-(e2)$) {};
}
\node[below left] at ($(O)$) {\tiny $u_1$};
\node[below left] at ($(O)-(e2)+(e1)$) {\tiny $u_2$};
\node[draw=red,circle,inner sep=1.5pt,fill=teal!70] at ($(O)$) {};
\node[draw=red,circle,inner sep=1.5pt,fill=teal!70] at ($(O)-(e2)+(e1)$) {};

\draw[black] (T) node {$\{u\in S_\sigma: \langle p,u\rangle=1\}$};

\end{tikzpicture}
\end{center}
\caption{The semigroup $S_\sigma$ in Example~\ref{dim3center_example} in $\{a=0\}$ and $\{a=1\}$.}
\label{dim3center_pict}
\end{figure}

Recall that $X = \{vw=zt\}$ is a quadratic cone of dimension~$3$, where $v=\chi^{(1,0,0)}$, $w=\chi^{(0,1,0)}$, $z=\chi^{(0,0,1)}$, and $t=\chi^{(1,1,-1)}$. Then
\[x = (v,w,z,t) \in Z(X) \;\Leftrightarrow\;
\left\{\begin{aligned}
\chi^u(x) = 0 &\;\text{ for any }\; u \in S_\sigma, \langle p,u\rangle > 0,\\
\chi^{(0,0,1)}(x) &= \chi^{(0,1,2)}(x),\\
\chi^{(0,1,0)}(x) &= \chi^{(0,2,1)}(x),
\end{aligned}\right. \;\Leftrightarrow\;
\left\{\begin{aligned}
v &= t = 0,\\
z &= wz^2,\\
w & = w^2z,
\end{aligned} 
\right.\]
whence $Z(X)$ consists of the hyperbola $\{v=t=0, \, wz=1\}$ of dimension~$1$ and the point $(0,0,0,0)$ of dimension~$0$. 
\end{example}

\section{Interplay between the center and idempotents} 

In this section, we study the connection of the center with the set of idempotents. 

Consider the action of the group of invertible elements $G_\chi$ of $X$ on $X$ by conjugation: $G_\chi \times X \to X$, $(g, x) \mapsto g * x * g^{-1}$. We call the orbits of this action \emph{conjugacy classes} of~$X$. 

\begin{proposition}
\label{idempcenter_prop}
In the notation of Theorem~\ref{classif_theor}, let $X = X_\sigma$ be an affine monoid of corank one. Then any irreducible component of $E(X)$ is a conjugacy class of $X$. 
In particular, irreducible components of $E(X)$ that are isomorphic to the affine line do not intersect $Z(X)$, and isolated points in $E(X)$ belong to $Z(X)$. 
\end{proposition}

\begin{proof}
If $x \in X$ is an idempotent, then $g*x*g^{-1}$ is an idempotent as well. So $E(X)$ is a union of conjugacy classes. Notice also that any conjugacy class is irreducible since the group $G_\chi$ is connected. 

By Proposition~\ref{idemp_geom_prop}(b), irreducible components of $E(X)$ do not intersect and any irreducible component of $E(X)$ is either an isolated point or isomorphic to the affine line. So any isolated point in $E(X)$ belongs to $Z(X)$. By definition, $x \in Z(X)$ if and only if its conjugacy class is trivial, so it remains to prove that any irreducible component $E'$ of $E(X)$ such that $E' \cong \AA^1$ is one conjugacy class. 

Consider a point $x \in E'$. Let us show that $x \notin Z(X)$. According to the proof of Proposition~\ref{idemp_geom_prop}, $x \in \overline{E_\gamma}$ for the face $\gamma$ from item (d) of Theorem~\ref{idemp_theor}. Since $x \in \overline{O_\gamma}$ we have $\chi^u(x) = 0$ for $u \notin \gamma^\perp$, and Theorem~\ref{idemp_theor} implies $\chi^u(x) = 1$ for $u \in \cone(\gamma, \rho)^\perp$, where $u \in S_\sigma$. 
We can assume that $e_1 \in \gamma^\perp$ and $e_2 \notin \gamma^\perp$. Consider $u \in \cone(\gamma, \rho)^\perp = \gamma^\perp \cap \rho^\perp \cap \sigma^\vee$ such that $\langle p,u\rangle = 1$. Then $\chi^{u+e_1}(x) = 1$ and $\chi^{u+e_2}(x) = 0$ by the above. So $x \notin Z(X)$ according to Proposition~\ref{centerX_prop}. 

Thus, points in $E' \cong \AA^1$ have nontrivial conjugacy classes, which do not intersect, are locally closed as orbits, and are subsets of~$E'$. It is possible only if $E'$ is one conjugacy class. 
\end{proof}

\begin{example}
\label{centerA2_example}
Let $X = \AA^2$ and the multiplication be $(x_1, x_2) * (y_1, y_2) = (x_1y_1, \, x_1^ay_2 + x_2)$, where $a > 0$, see Example~\ref{idempA2_example} and Figure~\ref{A2pict}. 
By Corollary~\ref{centerAn_cor}, we have \[Z(X) = \{x_2 = 0, \, x_1^a = 1\}.\] So the center $Z(X)$ consists of $a$ points $(\xi_i, 0)$, where $\{\xi_1, \ldots, \xi_a\}=\sqrt[a]{1}$. 

This agrees with Proposition~\ref{idempcenter_prop}: the isolated idempotent $(1,0)$ belongs to $Z(X)$, and the line of idempotents $(0,x_2)$, $x_2\in \KK$, does not intersect $Z(X)$. 
\end{example}

\begin{example}
\label{A4center_example}
Let $X = \AA^4$ and the multiplication be given by formula~\eqref{mult_A4ex_eq}, see Example~\ref{idempA4_example}. 
According to Corollary~\ref{centerAn_cor}, we have
\[Z(X) = \{x_4 = 0, \; x_2^bx_3^c = x_3^a\}.\]

Let us find the irreducible components of $Z(X)$, the closure $\overline{Z(G_\chi)}$ of the center of the group of invertible elements, and the set of idempotents $E(X)$. It is clear that $\overline{Z(G_\chi)}$ is a subvariety of $Z(X)$ and the unity is the unique idempotent in~$Z(G_\chi)$. We show that in our example $\overline{Z(G_\chi)} \ne Z(X)$, and idempotents can belong to different irreducible components depending on parameters. 

Denote $d = \gcd(b,a-c)$ and $\sqrt[d]{1} = \{\xi_1, \ldots, \xi_d\}$. Then $Z(X)$ has $d+1$ two-dimensional irreducible components; their equations depend on $a,b,c$, see Table~\ref{A4center_table}. 
\begin{table}[h]
\[\begin{array}{c|c|c|c}
& \text{case } c < a & \text{case } c = a & \text{case } c > a
\\\hline
&
\{x_3 = x_4 = 0\} \,\cup 
& \{x_3 = x_4 = 0\} \,\cup 
& \{x_3 = x_4 = 0\} \,\cup 
\\
Z(X) &
\cup\, \{x_2^\frac{b}{d}=\xi_1x_3^\frac{a-c}{d}\!, x_4 = 0\} \,\cup
& \cup\, \{x_2=\xi_1, x_4 = 0\} \,\cup 
& \cup\, \{x_2^\frac{b}{d}x_3^\frac{c-a}{d}\!=\xi_1, x_4 = 0\} \,\cup
\\
 &
\ldots & \ldots & \ldots
\\
& 
\cup\, \{x_2^\frac{b}{d}=\xi_dx_3^\frac{a-c}{d}\!, x_4 = 0\}
& \cup\, \{x_2=\xi_d, x_4 = 0\}
& \cup\, \{x_2^\frac{b}{d}x_3^\frac{c-a}{d}\!=\xi_d, x_4 = 0\}
\end{array}\]
\label{A4center_table}
\caption{Irreducible components of $Z(X)$ in Example~\ref{A4center_example}.}
\end{table}

The set of invertible elements equals $G_\chi = \{x_1, x_2, x_3 \ne 0\}$ and $Z(G_\chi) = Z(X) \cap G_\chi$. We call the plane $\{x_3 = x_4 = 0\}$ the external component of $Z(X)$ as it does not intersect~$G_\chi$. 

\begin{figure}[h]
\begin{center}
\tikzset{every picture/.style={line width=0.75pt}} 

\begin{tikzpicture}[x=0.75pt,y=0.75pt,yscale=-1,xscale=1]
    \tikzstyle{bluefill} = [fill=blue!20, draw = blue!80,opacity=0.85]
    \tikzstyle{violetfill} = [fill=violet!40, draw = violet!80,opacity=0.85]
    \tikzstyle{tealfill} = [fill=teal!15, draw = teal!80,opacity=0.85]

\coordinate (e1) at (-10,37);
\coordinate (e2) at (40,0);
\coordinate (e3) at (0,-40);


\coordinate (O) at (120,230);
\coordinate (Oxmax) at ($(O)+2*(e1)$);
\coordinate (Oymax) at ($(O)+2*(e2)$);
\coordinate (Ozmax) at ($(O)+2*(e3)$);
\coordinate (Oxmin) at ($(O)-1.5*(e1)$);
\coordinate (Oymin) at ($(O)-2*(e2)$);
\coordinate (Ozmin) at ($(O)-2.5*(e3)$);
\coordinate (A) at ($(O)  +1.5*(e1)$);
\coordinate (D) at ($(O)-0.5*(e1)$);

\coordinate (T) at ($(O)+1.5*(e1)-2*(e3)$);
\node at (T) {$b=4, \, a-c=6$};

\coordinate (B) at ($(A)-1.5*(e2)-1.3*(e3)$);
\coordinate (C) at ($(D)-1.5*(e2)-1.3*(e3)$);
\filldraw [violetfill] 
(O) -- (A)  .. controls +(-5,40) and +(3,-2) .. (B) -- (C) .. controls +(3,-2) and +(-5,40) .. (D) -- cycle;

\draw [->,color=black!50] (Oxmin) -- (Oxmax) node[below] {\scriptsize $x_1$};
\draw [color=black!50] (Ozmin) -- (O);
\draw[very thick] (O) .. controls +(-5,40) and +(3,-2) .. ($(O)-1.5*(e2)-1.3*(e3)$);

\coordinate (B) at ($(A)+1.5*(e2)-1.3*(e3)$);
\coordinate (C) at ($(D)+1.5*(e2)-1.3*(e3)$);
\filldraw [violetfill] 
(O) -- (A)  .. controls +(5,40) and +(-3,-2) .. (B) -- (C) node[right,violet] {\tiny $x_2^2=-x_3^3$} .. controls +(-3,-2) and +(5,40) .. (D) -- cycle;
\draw[very thick] (O) .. controls +(5,40) and +(-3,-2) .. ($(O)+1.5*(e2)-1.3*(e3)$);

\coordinate (B) at ($(A)-1.5*(e2)$);
\coordinate (C) at ($(D)-1.5*(e2)$);
\filldraw [tealfill] (O) -- (A) -- (B) -- (C) -- (D) -- cycle;
\coordinate (B) at ($(A)+1.5*(e2)$);
\coordinate (C) at ($(D)+1.5*(e2)$);
\filldraw [tealfill] (O) -- (A) -- (B) -- (C) node[right,teal] {\tiny $x_3=0$} -- (D) -- cycle;

\draw [->,color=black!50] (Oymin) -- (Oymax) node[below] {\scriptsize $x_2$};

\coordinate (B) at ($(A)-1.5*(e2)+1.3*(e3)$);
\coordinate (C) at ($(D)-1.5*(e2)+1.3*(e3)$);
\filldraw [bluefill] 
(O) -- (A)  .. controls +(-5,-40) and +(3,2) .. (B) -- (C) .. controls +(3,2) and +(-5,-40) .. (D) -- cycle;
\draw[very thick] (O) .. controls +(-5,-40) and +(3,2) .. ($(O)-1.5*(e2)+1.3*(e3)$);

\draw [->,color=black!50] (O) -- (Ozmax) node[left] {\scriptsize $x_3$};
\draw[very thick] ($(D)-0.02*(e2)$) -- ($(A)-0.02*(e2)$);

\coordinate (B) at ($(A)+1.5*(e2)+1.3*(e3)$);
\coordinate (C) at ($(D)+1.5*(e2)+1.3*(e3)$);
\filldraw [bluefill] 
(A)  .. controls +(5,-40) and +(-3,2) .. (B) -- (C) node[right,blue] {\tiny $x_2^2=x_3^3$} .. controls +(-3,2) and +(5,-40) .. (D);
\draw[very thick] (O) .. controls +(5,-40) and +(-3,2) .. ($(O)+1.5*(e2)+1.3*(e3)$);

\coordinate (e11) at ($0.9*(e1)$);
\coordinate (e22) at ($(e2)$);
\coordinate (e33) at ($1.12*(e3)$);
\node[draw=black,diamond,inner sep=1pt,fill=red] at (O) {};
\node[draw=black,diamond,inner sep=1pt,fill=red] at ($(O)+(e11)$) {};
\node[draw=teal!50,circle,inner sep=1pt,fill=red!50] at ($(O)+(e22)$) {};
\node[draw=red,circle,inner sep=1pt] at ($(O)+(e33)$) {};
\node[draw=teal,circle,inner sep=1pt,fill=red] at ($(O)+(e11)+(e22)$) {};
\node[draw=black,diamond,inner sep=1pt,fill=red] at ($(O)+(e22)+(e33)$) {};
\node[draw=red,circle,inner sep=1pt] at ($(O)+(e11)+(e33)$) {};
\node[draw=blue,star,inner sep=1pt,fill=red] at ($(O)+(e11)+(e22)+(e33)$) {};


\coordinate (O) at ($(O)+5*(e2)$);

\coordinate (Oxmax) at ($(O)+2*(e1)$);
\coordinate (Oymax) at ($(O)+2*(e2)$);
\coordinate (Ozmax) at ($(O)+2*(e3)$);
\coordinate (Oxmin) at ($(O)-1.5*(e1)$);
\coordinate (Oymin) at ($(O)-2*(e2)$);
\coordinate (Ozmin) at ($(O)-2.5*(e3)$);

\coordinate (M) at ($(O)-1*(e2)$);
\coordinate (A) at ($(M)+1.5*(e1)$);
\coordinate (D) at ($(M)-0.5*(e1)$);
\coordinate (B) at ($(A)-1.3*(e3)$);
\coordinate (C) at ($(D)-1.3*(e3)$);

\coordinate (T) at ($(O)+1.5*(e1)-2*(e3)$);
\node at (T) {$b=2, \, c = a$};

\filldraw [violetfill] 
(A)  -- (B) node[right,violet] {\tiny $x_2=-1$} -- (C) -- (D) -- cycle;
\draw[very thick] ($(M)-1.3*(e3)$) -- (M);

\draw [->,color=black!50] (Oxmin) -- (Oxmax) node[below] {\scriptsize $x_1$};
\draw [color=black!50] (Ozmin) -- (O);

\coordinate (M) at ($(O)+1*(e2)$);
\coordinate (A) at ($(M)+1.5*(e1)$);
\coordinate (D) at ($(M)-0.5*(e1)$);
\coordinate (B) at ($(A)-1.3*(e3)$);
\coordinate (C) at ($(D)-1.3*(e3)$);
\filldraw [bluefill] 
(A)  -- (B) -- (C) -- (D) -- cycle;
\draw[very thick] ($(M)-1.3*(e3)$) -- (M);

\coordinate (A) at ($(O)+1.5*(e1)$);
\coordinate (D) at ($(O)-0.5*(e1)$);
\coordinate (B) at ($(A)-1.5*(e2)$);
\coordinate (C) at ($(D)-1.5*(e2)$);
\filldraw [tealfill] (O) -- (A) -- (B) -- (C) -- (D) -- cycle;
\coordinate (B) at ($(A)+1.5*(e2)$);
\coordinate (C) at ($(D)+1.5*(e2)$);
\filldraw [tealfill] (O) -- (A) -- (B) node[right,teal] {\tiny $x_3=0$} -- (C) -- (D) -- cycle;

\draw [->,color=black!50] (Oymin) -- (Oymax) node[below] {\scriptsize $x_2$};

\coordinate (M) at ($(O)-1*(e2)$);
\coordinate (A) at ($(M)+1.5*(e1)$);
\coordinate (D) at ($(M)-0.5*(e1)$);
\coordinate (B) at ($(A)+1.3*(e3)$);
\coordinate (C) at ($(D)+1.3*(e3)$);
\filldraw [violetfill] 
(A)  -- (B) -- (C) -- (D) -- cycle;
\draw[very thick] ($(M)+1.3*(e3)$) -- (M);

\draw [->,color=black!50] (O) -- (Ozmax) node[left] {\scriptsize $x_3$};

\coordinate (M) at ($(O)+1*(e2)$);
\coordinate (A) at ($(M)+1.5*(e1)$);
\coordinate (D) at ($(M)-0.5*(e1)$);
\coordinate (B) at ($(A)+1.3*(e3)$);
\coordinate (C) at ($(D)+1.3*(e3)$);
\filldraw [bluefill] 
(A)  -- (B) -- (C) node[right,blue] {\tiny $x_2=1$} -- (D) -- cycle;
\draw[very thick] ($(M)+1.3*(e3)$) -- (M);

\draw[very thick] ($(O)-1*(e2)-0.5*(e1)$) -- ($(O)-1*(e2)+1.5*(e1)$);
\draw[very thick] ($(O)+1*(e2)-0.5*(e1)$) -- ($(O)+1*(e2)+1.5*(e1)$);

\node[draw=teal,circle,inner sep=1pt,fill=red] at (O) {};
\node[draw=teal,circle,inner sep=1pt,fill=red] at ($(O)+(e1)$) {};
\node[draw=black,diamond,inner sep=1pt,fill=red] at ($(O)+(e2)$) {};
\node[draw=red,circle,inner sep=1pt] at ($(O)+(e3)$) {};
\node[draw=black,diamond,inner sep=1pt,fill=red] at ($(O)+(e1)+(e2)$) {};
\node[draw=black,diamond,inner sep=1pt,fill=red] at ($(O)+(e2)+(e3)$) {};
\node[draw=red,circle,inner sep=1pt] at ($(O)+(e1)+(e3)$) {};
\node[draw=blue,star,inner sep=1pt,fill=red] at ($(O)+(e1)+(e2)+(e3)$) {};


\coordinate (O) at ($(O)+5*(e2)$);
\coordinate (Oxmax) at ($(O)+2*(e1)$);
\coordinate (Oymax) at ($(O)+2*(e2)$);
\coordinate (Ozmax) at ($(O)+2*(e3)$);
\coordinate (Oxmin) at ($(O)-1.5*(e1)$);
\coordinate (Oymin) at ($(O)-2*(e2)$);
\coordinate (Ozmin) at ($(O)-2.5*(e3)$);
\coordinate (A) at ($(O)  +1.5*(e1)$);
\coordinate (D) at ($(O)-0.5*(e1)$);

\coordinate (T) at ($(O)+1.5*(e1)-2*(e3)$);
\node at (T) {$b=4, \, c-a=6$};

\coordinate (Ado) at ($(O)+1.5*(e1)-1.3*(e3)$);
\coordinate (Ddo) at ($(O)-0.5*(e1)-1.3*(e3)$);
\coordinate (Odo) at ($(O)         -1.3*(e3)$);

\coordinate (A1) at ($(Ado)-0.5*(e2)$);
\coordinate (D1) at ($(Ddo)-0.5*(e2)$);
\coordinate (O1) at ($(Odo)-0.5*(e2)$);
\coordinate (B) at ($(A)-1.5*(e2)-0.5*(e3)$);
\coordinate (C) at ($(D)-1.5*(e2)-0.5*(e3)$);
\filldraw [violetfill] 
(A1)  .. controls +(-5,-30) and +(3,2) .. (B) -- (C) .. controls +(3,2) and +(-5,-30) .. (D1) -- cycle;

\draw [->,color=black!50] (Oxmin) -- (Oxmax) node[below] {\scriptsize $x_1$};
\draw [color=black!50] (Ozmin) -- (O);
\draw[very thick] (O1) .. controls +(-5,-30) and +(3,2) .. ($(O)-1.5*(e2)-0.5*(e3)$);

\coordinate (A1) at ($(Ado)+0.5*(e2)$);
\coordinate (D1) at ($(Ddo)+0.5*(e2)$);
\coordinate (O1) at ($(Odo)+0.5*(e2)$);
\coordinate (B) at ($(A)+1.5*(e2)-0.5*(e3)$);
\coordinate (C) at ($(D)+1.5*(e2)-0.5*(e3)$);
\filldraw [violetfill] 
(A1) node[right,violet] {\tiny $x_2^2x_3^3=-1$} .. controls +(5,-30) and +(-3,2) .. (B) -- (C) .. controls +(-3,2) and +(5,-30) .. (D1) -- cycle;
\draw[very thick] (O1) .. controls +(5,-30) and +(-3,2) .. ($(O)+1.5*(e2)-0.5*(e3)$);

\coordinate (B) at ($(A)-1.5*(e2)$);
\coordinate (C) at ($(D)-1.5*(e2)$);
\filldraw [tealfill] (O) -- (A) -- (B) -- (C) -- (D) -- cycle;
\coordinate (B) at ($(A)+1.5*(e2)$);
\coordinate (C) at ($(D)+1.5*(e2)$);
\filldraw [tealfill] (O) -- (A) -- (B) -- (C) node[right,teal] {\tiny $x_3=0$} -- (D) -- cycle;

\draw [->,color=black!50] (Oymin) -- (Oymax) node[below] {\scriptsize $x_2$};

\coordinate (Aup) at ($(O)+1.5*(e1)+1.3*(e3)$);
\coordinate (Dup) at ($(O)-0.5*(e1)+1.3*(e3)$);
\coordinate (Oup) at ($(O)         +1.3*(e3)$);

\coordinate (A1) at ($(Aup)-0.5*(e2)$);
\coordinate (D1) at ($(Dup)-0.5*(e2)$);
\coordinate (O1) at ($(Oup)-0.5*(e2)$);
\coordinate (B) at ($(A)-1.5*(e2)+0.5*(e3)$);
\coordinate (C) at ($(D)-1.5*(e2)+0.5*(e3)$);
\filldraw [bluefill] 
(A1)  .. controls +(-5,30) and +(3,-2) .. (B) -- (C) .. controls +(3,-2) and +(-5,30) .. (D1) -- cycle;
\draw[very thick] (O1) .. controls +(-5,30) and +(3,-2) .. ($(O)-1.5*(e2)+0.5*(e3)$);

\draw [->,color=black!50] (O) -- (Ozmax) node[left] {\scriptsize $x_3$};

\coordinate (A1) at ($(Aup)+0.5*(e2)$);
\coordinate (D1) at ($(Dup)+0.5*(e2)$);
\coordinate (O1) at ($(Oup)+0.5*(e2)$);
\coordinate (B) at ($(A)+1.5*(e2)+0.5*(e3)$);
\coordinate (C) at ($(D)+1.5*(e2)+0.5*(e3)$);
\filldraw [bluefill] 
(A1)  .. controls +(5,30) and +(-3,-2) .. (B) -- (C) .. controls +(-3,-2) and +(5,30) .. (D1) node[right,blue] {\tiny $x_2^2x_3^3=1$} -- cycle;
\draw[very thick] (O1) .. controls +(5,30) and +(-3,-2) .. ($(O)+1.5*(e2)+0.5*(e3)$);

\coordinate (e1) at ($0.9*(e1)$);
\coordinate (e2) at ($0.81*(e2)$);
\coordinate (e3) at ($0.74*(e3)$);
\node[draw=teal,circle,inner sep=1pt,fill=red] at (O) {};
\node[draw=teal,circle,inner sep=1pt,fill=red] at ($(O)+(e1)$) {};
\node[draw=teal!50,circle,inner sep=1pt,fill=red!50] at ($(O)+(e2)$) {};
\node[draw=red,circle,inner sep=1pt] at ($(O)+(e3)$) {};
\node[draw=teal,circle,inner sep=1pt,fill=red] at ($(O)+(e1)+(e2)$) {};
\node[draw=black,diamond,inner sep=1pt,fill=red] at ($(O)+(e2)+(e3)$) {};
\node[draw=red,circle,inner sep=1pt] at ($(O)+(e1)+(e3)$) {};
\node[draw=blue,star,inner sep=1pt,fill=red] at ($(O)+(e1)+(e2)+(e3)$) {};

\end{tikzpicture}
\end{center}
\caption{The center and idempotents in $X$ in Example~\ref{A4center_example} in $\{x_4=0\}$.}
\label{A4center_pict}
\end{figure}
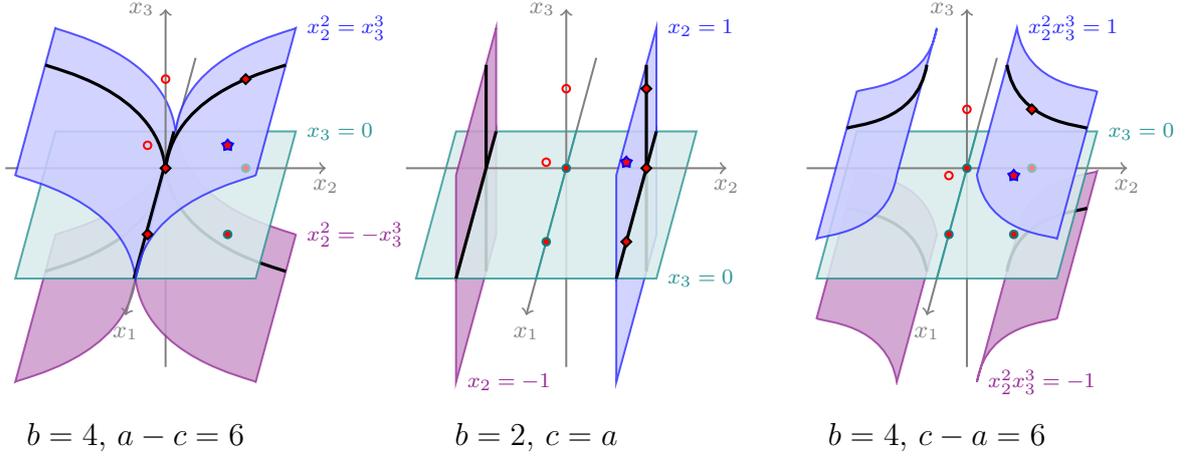
\end{example}

In Figure~\ref{A4center_pict}, one can see the illustration in the subspace $\{x_4 = 0\}$ if $d = 2$; suppose $b=4, \, a-c=\pm 6$ or $b=2, \, a-c=0$. Irreducible components of $Z(X)$ have different colors, the external one is horizontal. Note that $\overline{Z(G_\chi)}$ is the union of non-external irreducible components of $Z(X)$. The external component intersects the non-external ones if $c < a$ or $c = a$. The group center $Z(G_\chi)$ is the union of non-external components without the boundary of $Z(G_\chi)$, which is colored in black. 

Idempotents are figured as small circles, squares and stars. Recall that according to Example~\ref{idempA4_example} the set of idempotents consists of two lines $(0,0,1,x_4)$ and $(1,0,1,x_4)$, $x_4 \in \KK$, and six isolated points 
\[(0, 0, 0, 0), (0, 1, 0, 0), (0, 1, 1, 0), (1, 0, 0, 0), (1, 1, 0, 0), (1, 1, 1, 0).\]
In the subspace $\{x_4 = 0\}$ we see six isolated points that belong to~$Z(X)$ and two points of two lines that do not. This agrees with Proposition~\ref{idempcenter_prop}. The distribution of six isolated points between the irreducible components depends on parameters $a,c$, see Figure~\ref{A4center_pict} and Table~\ref{A4idempcenter_table}. 

\begin{table}[h]
\[\begin{array}{l|c|c|c}
\text{The number of isolated idempotents} & \text{case } c < a & \text{case } c = a & \text{case } c > a
\\\hline
\begin{tikzpicture}\node[draw=blue,star,inner sep=1pt,fill=red,line width=0.75pt] at (0,0) {};\end{tikzpicture} \quad
\text{In } Z(G_\chi) \text{ (the unity)} & 1 & 1 & 1
\\
\begin{tikzpicture}\node[draw=black,diamond,inner sep=1pt,fill=red,line width=0.75pt] at (0,0) {};\end{tikzpicture} \quad
\text{On the boundary of } Z(G_\chi) & 3 & 3 & 1
\\
\begin{tikzpicture}\node[draw=teal,circle,inner sep=1pt,fill=red,line width=0.75pt] at (0,0) {};\end{tikzpicture} \quad
\text{In } Z(X)\setminus \overline{Z(G_\chi)} & 2 & 2 & 4
\end{array}\]
\caption{The interplay of $E(X)$, $Z(X)$, and $Z(G_\chi)$ in Example~\ref{A4center_example}.}
\label{A4idempcenter_table}
\end{table}


%




\end{document}